\documentclass[11pt]{article}

\usepackage{amsmath, amsthm, amssymb, tikz, hyperref, enumerate,array}
\usepackage{amsfonts}
\usepackage{fullpage}
\usepackage[enableskew,vcentermath]{youngtab}
\usepackage{tikz-cd}
\usetikzlibrary{decorations.pathreplacing,patterns}

\newcommand\beq{\begin{equation}}
\newcommand\eeq{\end{equation}}
\newcommand\bce{\begin{center}}
\newcommand\ece{\end{center}}
\newcommand\bea{\begin{eqnarray}}
\newcommand\eea{\end{eqnarray}}
\newcommand\ba{\begin{array}}
\newcommand\ea{\end{array}}
\newcommand\ben{\begin{enumerate}}
\newcommand\een{\end{enumerate}}
\newcommand\bit{\begin{itemize}}
\newcommand\eit{\end{itemize}}
\newcommand\brr{\begin{array}}
\newcommand\err{\end{array}}
\newcommand\bt{\begin{tabular}}
\newcommand\et{\end{tabular}}

\newcommand\ul{\underline}

\newcommand\s{{\lambda/\mu}}

\renewcommand\S{{\mathcal S}}
\newcommand\U{{\mathcal U}}
\newcommand\AAA{{\mathcal A}}
\newcommand\PPP{{\mathcal P}}
\newcommand\QQQ{{\mathcal Q}}

\DeclareMathOperator\SYT{SYT}

\DeclareMathOperator\rev{rev}
\DeclareMathOperator\inte{int}

\newcommand\x{{\mathbf x}}
\newcommand\ten{10}

\newcommand\blank{$ $}
\newcommand{\rone}{\color{red}{\mathbf 1}}
\newcommand{\rtwo}{\color{red}{\mathbf 2}}
\newcommand{\rthree}{\color{red}{\mathbf 3}}
\newcommand{\rfour}{\color{red}{\mathbf 4}}
\newcommand{\rfive}{\color{red}{\mathbf 5}}
\newcommand{\rsix}{\color{red}{\mathbf 6}}
\newcommand{\rseven}{\color{red}{\mathbf 7}}
\newcommand{\reight}{\color{red}{\mathbf 8}}
\newcommand{\rnine}{\color{red}{\mathbf 9}}

\DeclareMathOperator\coeff{coeff}

\newcommand\ele{\varepsilon}

\newcommand{\bbz}{\mathbb{Z}}

\newcommand{\bij}{\phi}

\DeclareMathOperator\Des{Des}
\DeclareMathOperator\des{des}
\DeclareMathOperator\cDes{cDes}
\DeclareMathOperator\cdes{cdes}

\newcommand{\Jdt}{\operatorname{jdt}}
\newcommand{\jdt}{\widetilde{\rm{jdt}}}

\DeclareMathOperator\Peak{\Lambda} 
\DeclareMathOperator\cPeak{c\Lambda}
\newcommand\cPeakA{\cPeak_\AAA}
\newcommand\cPeakP{\cPeak_\PPP}
\newcommand\cPeakm{\cPeak_{\PPP,m}}

\newcommand{\NN}{\mathbb{N}}
\newcommand{\ZZ}{\mathbb{Z}}

\newcommand{\card}[1]{{\lvert #1 \rvert}}   

\newcommand{\TTT}{{\mathcal{T}}}
\newcommand{\hTTT}{{\hat{\TTT}}}

\newcommand{\multiset}[1]{\{\!\{#1\}\!\}}

\newlength{\mysizetiny}
\newlength{\mysizesmall}
\newlength{\mysize}
\newlength{\mysizelarge}
\theoremstyle{plain}
\newtheorem{theorem}{Theorem}[section]
\newtheorem{proposition}[theorem]{Proposition}
\newtheorem{lemma}[theorem]{Lemma}
\newtheorem{corollary}[theorem]{Corollary}

\newtheorem{problem}[theorem]{Problem}

\theoremstyle{definition}
\newtheorem{definition}[theorem]{Definition}
\newtheorem*{example}{Example}

\theoremstyle{remark}

\newtheorem{remark}[theorem]{Remark}

\numberwithin{figure}{section}

\title{Cyclic descents for near-hook and two-row shapes}

\author{Ron M.\ Adin~\thanks{Department of Mathematics, Bar-Ilan University, Ramat-Gan 52900, Israel. E-mail: {\tt radin@math.biu.ac.il}} 
\and
Sergi Elizalde~\thanks{Department of Mathematics, 6188 Kemeny Hall, Dartmouth College, Hanover, NH 03755, USA. E-mail: {\tt sergi.elizalde@dartmouth.edu}. 
} 
\and 
Yuval Roichman~\thanks{Department of Mathematics, Bar-Ilan University, Ramat-Gan 52900, Israel. E-mail: {\tt yuvalr@math.biu.ac.il}.}}

\date{}

\begin{document}

\maketitle

\begin{abstract}
A notion of cyclic descents on standard Young tableaux (SYT) of rectangular shape
was introduced by Rhoades, and extended to certain 
skew shapes by the last two authors.
The cyclic descent set restricts to the usual descent set when the largest value is ignored, and has the property that
the number of SYT of a given shape with a given cyclic descent set $D$ is invariant under cyclic shifts of the entries of $D$. 
Following these results, the existence of cyclic descent sets for standard Young tableaux of any skew shape other than a ribbon was conjectured by the authors, 
and recently proved by Adin, Reiner and Roichman.
Unfortunately, the proof does not provide a natural definition of the cyclic descent set for a specific tableau. 

In this paper we explicitly describe cyclic descent sets and resulting generating functions for SYT of (possibly skew) shapes which either have exactly two rows or are near-hooks, i.e., are one cell away from a hook.
Our definition provides a constructive combinatorial proof of the existence of cyclic descent sets for these shapes, and coincides with that of Rhoades for two-row rectangular shapes.
We also show that cyclic descent sets for near-hook shaped tableaux are unique. 
\end{abstract}



\tableofcontents

\section{Introduction}\label{sec:intro}

Let $[n]=\{1,2,\dots,n\}$ and let $\S_n$ denote the symmetric group on $[n]$. Recall that the {\em descent set} of a permutation $\pi\in\S_n$ is
\[
\Des(\pi) := \{i \,:\, 1\le i\le n-1,\, \pi(i)>\pi(i+1)\}.
\]
The {\em cyclic descent set} of a permutation was introduced by
Cellini~\cite{Cellini} and further studied in~\cite{Petersen,
Dilks}. 
It is defined, for $\pi\in\S_n$, by
\[
\cDes(\pi) := \begin{cases}
\Des(\pi) \sqcup \{n\}, &  \text{if } \pi(n)> \pi(1),\\
\Des(\pi), & \text{otherwise}.
\end{cases}
\]
Clearly, $\cDes(\pi)\cap[n-1]=\Des(\pi)$.
Moreover, for any $D \subseteq [n]$, letting $1+D$ be the subset of $[n]$ obtained from $D$ by adding $1 \pmod n$ to each element, the number of permutations in $\S_n$ with cyclic descent set $D$ is equal to the number of permutations with cyclic descent $1+D$. Equivalently, the multiset $\multiset{\cDes(\pi):\pi\in\S_n}$ is closed under cyclic rotation modulo $n$. Throughout this paper, double curly braces will be used to denote multisets.

\medskip

Another important family of combinatorial objects for which there
is a well-studied notion of descent set are standard Young
tableaux (SYT). 
Let $\lambda/\mu$ denote a skew shape, where $\lambda$ and $\mu$ are partitions such that
the Young diagram of $\mu$ is contained in that of $\lambda$.
Let $\SYT(\lambda/\mu)$ denote the set of standard Young tableaux of
shape $\lambda/\mu$. For a straight shape $\lambda=(\lambda_1,\lambda_2,\dots)$, we will write $\SYT(\lambda_1,\lambda_2,\dots)$ instead of $\SYT((\lambda_1,\lambda_2,\dots))$ for simplicity.
We draw tableaux in English notation, as in Figure~\ref{fig:SYT}.
The {\em descent set} of $T \in \SYT(\lambda/\mu)$ is
\[
\Des(T) := \{i\in[n-1] \,:\, i+1 \text{ is in a lower row than $i$ in $T$}\}.
\]
For example, the descent set of the SYT in Figure~\ref{fig:SYT} is $\{1,4,7,8\}$.

\begin{figure}[htb]
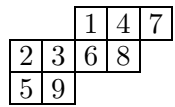

$$\young(::147,2368,59)$$
\caption{A SYT of shape $(5,4,2)/(2)$.} 
\label{fig:SYT}
\end{figure}

It is natural to ask whether an appropriate notion of {\em cyclic descent set} exists for $\SYT(\lambda/\mu)$. 
When $\lambda/\mu$ is a straight rectangular shape, such a notion was introduced by Rhoades~\cite{Rhoades} (who called it {\em extended descent set}); see also~\cite{PS}. 
The following theorem reformulates~\cite[Lemma 3.3]{Rhoades}. 
Let $2^{[n]}$ denote the collection of all subsets of $[n]$, and let $m^{n/m}$ denote the Young diagram of rectangular shape with $n/m$ rows and $m$ columns.

\begin{theorem}[\cite{Rhoades}]\label{Rhoades_thm}
Let $m$ be a divisor of $n$. There exists a map $\cDes:\SYT(m^{n/m})\to 2^{[n]}$
such that
\begin{enumerate}[(i)]
\item 
$\cDes(T)\cap[n-1]=\Des(T)$ for every $T\in\SYT(m^{n/m})$;
\item 
the multiset $\multiset{\cDes(T): T\in \SYT(m^{n/m})}$ is closed under
cyclic rotation modulo $n$.
\end{enumerate}
\end{theorem}

Rhoades' definition of $\cDes$ on $\SYT(m^{n/m})$ states that $n \in \cDes(T)$ if and only if $n-1$ is a descent of the SYT obtained from $T$ by Sch\"utzenberger promotion. 

An analogous phenomenon for certain disconnected shapes has recently been discovered in~\cite{ER16}. 
For a partition $\lambda$ of $n-1$, let $\lambda^\Box$ be the skew shape obtained from the Young diagram
of shape $\lambda$ by placing a disconnected cell at its upper right corner. 

\begin{theorem}[\cite{ER16}]\label{ER_thm}
Let $\lambda$ be a partition of $n-1$. There exists a map $\cDes:\SYT(\lambda^\Box)\to 2^{[n]}$ such that
\begin{enumerate}[(i)]
\item $\cDes(T)\cap[n-1]=\Des(T)$ for every $T\in\SYT(\lambda^\Box)$;
\item the multiset $\multiset{\cDes(T): T\in \SYT(\lambda^\Box)}$ is closed under cyclic rotation modulo $n$.
\end{enumerate}
\end{theorem}

The definition of the cyclic descent set in this case involves a
{\it jeu-de-taquin}-type straightening algorithm~\cite[Definition 7.1]{ER16}; see Section~\ref{sec:prel2} below.

\medskip

Motivated by Theorems~\ref{Rhoades_thm} and~\ref{ER_thm} and by the properties of cyclic descent sets of permutations, we define the key concept of this paper.

\begin{definition}\label{def:extension}
For a (possibly skew) Young diagram $\lambda/\mu$ with $n$ cells, a {\em cyclic descent
map for $\lambda/\mu$} is a map $\cDes:\SYT(\lambda/\mu)\to 2^{[n]}$  such that
\begin{enumerate}[(i)]
\item  
$\cDes(T)\cap[n-1]=\Des(T)$ for every $T\in\SYT(\lambda/\mu)$; and
\item 
the multiset $\multiset{\cDes(T) \,:\, T \in \SYT(\lambda/\mu)}$ is closed under cyclic rotation modulo $n$.
\end{enumerate}
\end{definition}

For a given cyclic descent map $\cDes$ for $\lambda/\mu$
and a tableau $T\in\SYT(\lambda/\mu)$, denote by $\cDes(T)$ the {\em cyclic descent set} of $T$, and call its elements {\em cyclic descents}.

\begin{example}
A cyclic descent map for $\lambda =(3,2)$ is given by
\[
\cDes\left(\young(134,25)\right)=\{1,4\}, \ \ \
\cDes\left(\young(125,34)\right)=\{2,5\}, \ \ \
\cDes\left(\young(135,24)\right)=\{3,1\},
\]
\[
\cDes\left(\young(124,35)\right)=\{4,2\}, \ \ \
\cDes\left(\young(123,45)\right)=\{5,3\}.
\]
\end{example}

Note that, by condition (i) in Definition~\ref{def:extension}, a cyclic descent 
map for $\lambda/\mu$ is uniquely determined by specifying, for each tableau $T \in \SYT(\lambda/\mu)$, whether or not $n$ is a cyclic descent of $T$. Theorems~\ref{Rhoades_thm} and~\ref{ER_thm} above state that there exist cyclic descent 
maps for the shapes $(m^{n/m})$ and $\lambda^\Box$, respectively.

Recall that {\em (connected) ribbons} are 
connected skew shapes containing no $2 \times 2$ rectangle. 
The following result is relatively straightforward, and follows from the proof of~\cite[Theorem 1.1]{ARR}, restated as Theorem~\ref{conj1} below.

\begin{proposition}\label{prop:ribbons}
If $\gamma$ is a connected ribbon with $n$ cells and height $1<h<n$,
then there is no cyclic descent map for $\gamma$ as in Definition~\ref{def:extension}.
\end{proposition}

The restriction $1 < h < n$ in this statement is due to a slight difference (concerning the ``non-Escher property'') between Definition~\ref{def:extension} above and the corresponding definition in~\cite{ARR} (reproduced as Definition~\ref{def:cDes} below).

In an early version of this paper we conjectured that all skew shapes, with the exception of connected ribbons,
have a cyclic descent map. This conjecture has recently been proved in~\cite{ARR},
using nonnegativity properties of Postnikov's toric Schur polynomials. 
The proof relies on an interpretation of the fiber sizes of the cyclic descent map as coefficients in the expansion of a certain toric Schur function in the Schur basis; equivalently, on a new combinatorial interpretation of certain Gromov--Witten invariants as fiber sizes of the cyclic descent map.
Unfortunately, the proof does not provide a natural definition of the cyclic descent map. 

The current paper is intimately related to~\cite{ARR}.
While \cite{ARR} presents a general  algebraic point of view, the approach in the current paper is more combinatorial.
The main goal of this paper is to explicitly describe cyclic descent maps for certain families of shapes:
near-hooks, i.e., shapes which are one cell away from hooks, 
and two-row shapes, both straight and skew.
In the special case of two-row rectangular shapes, our definition coincides with that of Rhoades~\cite{Rhoades}.
Our description provides a constructive combinatorial proof of the existence of cyclic descent extensions for these shapes;
for near-hooks we also prove that the cyclic descent map is unique.
Two additional goals of the paper are the 
explicit description of a $\bbz$-action on the tableaux of each shape that cyclically rotates their cyclic descent sets, 
and explicit formulas (generating functions) for the distribution of cyclic descents.

\medskip

Let us give an overview of the rest of the paper, according to the three goals outlined above.
We start with the preliminary Section~\ref{sec:Prel}, introducing some necessary background regarding {\em jeu de taquin} and its generalizations, which play an important role in the description of certain cyclic descent maps and $\bbz$-actions. 

Cyclic descent maps are explicitly described for various shapes: 
near-hooks, which are hooks plus or minus one cell (Section~\ref{sec:uniqueness});
strips, which generalize one type of near-hooks, that of hooks minus the corner cell (Section~\ref{sec:strips}); 
straight two-row shapes (Section~\ref{sec:2rows});
and skew two-row shapes, for which we present two different cyclic descent maps
(Section~\ref{sec:skew}).
As an appetizer, we present two of the results,
one from Section~\ref{sec:uniqueness} (a special case of Corollary~\ref{t.n_in_cDes})
and one from Section~\ref{sec:2rows}.
Here $T_{i,j}$ denotes the entry in row $i$ (from the top) and column $j$ (from the left) of $T$.

\begin{theorem}\label{thm:hooks_plus1_intro}
For every $2 \le k \le n-2$, the shape $(n-k,2,1^{k-2})$ has a unique cyclic descent map.
Given $T \in \SYT(n-k,2,1^{k-2})$, this map is specified by letting $n \in \cDes(T)$ if and only if the entry $T_{2,2}-1$ appears in the first column of $T$.
\end{theorem}

\begin{theorem}\label{thm:2rows_intro}
For every $2 \le k\le n/2$ there exists a cyclic descent map for the shape $(n-k,k)$. 
Given $T\in \SYT(n-k,k)$, such a map is specified by letting $n\in\cDes(T)$ if and only if both of the following conditions hold:
\begin{enumerate}
\item 
the last two entries in the second row of $T$ are consecutive, that is, $T_{2,k}=T_{2,k-1} + 1$;
\item 
for every $1<i<k$,  $T_{2,i-1} > T_{1,i}$.
\end{enumerate}
\end{theorem}

The cyclic descent maps described in Theorem~\ref{thm:hooks_plus1_intro} and Theorem~\ref{thm:2rows_intro} coincide on the shape $(n-2,2)$.

Explicit $\bbz$-actions showing that condition (ii) in Definition~\ref{def:extension} is satisfied are described for
hooks plus one internal cell (Section~\ref{sec:hooks_plus_action});
strips, including hooks minus a corner cell (Section~\ref{sec:near-hooks_to_strips});
and straight two-row shapes (Section~\ref{sec:2rows_action}). 
No explicit descriptions of such actions are known for skew two-row shapes.

The fiber sizes of the cyclic descent map on near-hook and straight two-row shapes are described in Theorems~\ref{t.near-hook:fibers} and \ref{t.gf_cdes_2}, respectively. 
The associated generating functions are computed in Sections~\ref{sec:GF-near-hooks} (for near-hooks), \ref{sec:words} (for certain strips), and~\ref{subsec:2rows_gf} (for straight two-row shapes).

Finally, Section~\ref{sec:final} contains a summary (see Table~\ref{tab:cDes}) of the definitions of cyclic descents for the various shapes considered in this paper, and presents some open problems.

\subsection*{Acknowledgements.} 
The authors thank Connor Ahlbach, Alessandro Iraci, Oliver Pechenik, Jim Propp, Amitai Regev, Tom Roby, and Josh Swanson for useful discussions.
The first author thanks the Israel Institute for Advanced Studies for its hospitality during part of the work on this paper.
The first and third authors were partially supported by an MIT-Israel MISTI grant.
The second author was partially supported by Simons Foundation grant \#280575.

\section{Preliminaries}\label{sec:Prel}

\subsection{Jeu de taquin}\label{sec:jdt} 

An important tool in this paper is 
the {\it jeu de taquin} (jdt) construction; see \cite[Appendix A1]{EC2} for a detailed description.
For any SYT $T$ of skew shape, denote by $\Jdt(T)$
the SYT of straight shape obtained from $T$ by performing a sequence of jdt slides.
This tableau is unique by~\cite[Theorem A1.2.4]{EC2}. The following is a well-known property of jdt.

\begin{lemma}[{\cite[Lemma 3.2]{Doran}}]\label{lem:jdtDes}
For any SYT $T$ of skew shape,
\[
\Des(\Jdt(T))=\Des(T).
\]
\end{lemma}

A Young tableau is semistandard   if the entries weakly increase along each row and strictly increase down each column.
 The {\em content} of a semistandard  Young tableau $T$ is the (eventually zero) sequence whose $i$th component is the number of $i$'s 
 in $T$. The {\em reverse reading word} of $T$ is obtained by reading the entries in each row from right to left, and reading the rows from top to bottom. A {\em lattice permutation} is a sequence of positive integers $a_1a_2\dots a_n$ such that every prefix contains at least as many $i$'s as $i+1$'s, for all $i\in \NN$.

Comparing different formulations of the  Littlewood--Richardson rule~\cite[Theorems A1.3.1 and A1.3.3]{EC2}, we obtain the following.

\begin{theorem}
\label{LR_jdt}
Fix $P\in\SYT(\nu)$, and let
\[
c^\lambda_{\mu\nu}:=
|\{T\in \SYT(\lambda/\mu):\ \Jdt(T)=P \}|.
\]
Then $c^\lambda_{\mu\nu}$
is equal to the number of semistandard Young tableaux of shape $\lambda/\mu$
 and content $\nu$ whose reverse reading word is a lattice permutation.
\end{theorem}

The following property of $\Jdt$ for two-row tableaux will be used in Sections~\ref{sec:2rows} and~\ref{sec:skew}.

\begin{lemma}\label{lem:shapes}
For every $n$, $k$ and $m$ with $0\le m\le k<n$ and $2k \le n+m$,
\[
\{\Jdt(T):\ T\in\SYT((n-k+m,k)/(m))\}=
\bigcup_{d=k-m}^{\min\{k,n-k\}} \SYT(n-d,d).
\]
\end{lemma}

\begin{proof}
Let $\lambda=(n-k+m,k)$ and $\mu=(m)$. For $T\in\SYT(\lambda/\mu)$, the tableau $\Jdt(T)$ has shape $\nu=(n-d,d)$ for some $d$. By Theorem~\ref{LR_jdt}, $c^\lambda_{\mu\nu}=1$ if 
$k-m\le d\le \min\{k,n-k\}$ and $c^\lambda_{\mu\nu}=0$ otherwise. 
Thus, every tableau $P$ of shape $(n-d,d)$ with $k-m\le d\le \min\{k,n-k\}$  
appears exactly once as the image under $\Jdt$ of some $T \in \SYT(\lambda/\mu)$.
\end{proof}

\subsection{Generalized jeu de taquin and cyclic descents}\label{sec:prel2}

For a set $D\subseteq[n]$ and an integer $k$, let $k+D$ denote the subset of $[n]$ obtained by adding $k \pmod n$ to each element of $D$.  
Similarly, for $T \in \SYT(\lambda/\mu)$ with $n$ cells, let $k+T$ be the tableau with entries in $[n]$ obtained from $T$ by adding $k\pmod n$ to each entry.
Note that $k+T$ is not standard in general. 
Define $\jdt(k+T)$ to be the standard tableau obtained by applying the following procedure, based on {\it jeu de taquin}, to $k+T$:
\medskip

Set $T_0 = k+T$, and repeat the following {\em elementary step} $\ele$ until $T_0$ is a standard tableau.
\begin{itemize}
\item[($\ele$)]  
Let $i$ be the minimal entry in $T_0$ with the property that the entries immediately above and to the left of it are not both smaller than $i$;
switch $i$ with the larger of these two entries, and let $T_0$ be the resulting tableau.
\end{itemize}

In our notation, the promotion operation%
\footnote{There are two definitions of the promotion operation in the literature. Whereas in~\cite{Rhoades,ER16} promotion is defined as $T\mapsto\jdt(1+T)$, we have chosen to follow here the definition $T\mapsto\jdt(-1+T)$ from~\cite{PPR,StanPE}.} 
can be described as
$T\mapsto\jdt(-1+T)$. We denote by $p T=\jdt(-1+T)$ the image of $T$ under promotion. 

Rhoades' result (Theorem~\ref{Rhoades_thm} above) may be now reformulated as follows.
\begin{theorem}[{\cite[Lemma 3.3]{Rhoades}}]\label{Rhoades_thm2}
Let $m$ be a divisor of $n$.  There exists a cyclic descent map for the shape $(m^{n/m})$, defined for $T \in \SYT(m^{n/m})$ by letting $n\in\cDes(T)$ if and only if $n-1\in \Des(p T)$.
\end{theorem}

With the definition of $\cDes$ on $\SYT(m^{n/m})$ given by Theorem~\ref{Rhoades_thm}, the inverse promotion operation $p^{-1}:\SYT(m^{n/m})\to\SYT(m^{n/m})$ cyclically shifts $\cDes$, and it generates a $\bbz_n$-action on $\SYT(m^{n/m})$.

For $\lambda\vdash n-1$, let $\lambda^\Box$ be the skew shape obtained from the Young diagram of shape $\lambda$ by placing a disconnected cell at its upper right corner. 
The following theorem reformulates \cite[Proposition 5.3]{ER16}.

\begin{theorem}[{\cite{ER16}}]\label{thm:cDes_lambda_box}
Let $\lambda\vdash n-1$. 
There exists a cyclic descent map for $\lambda^\Box$, defined for $T\in \SYT(\lambda^\Box)$ by letting $n\in\cDes(T)$ if and only if
$n$ is strictly north of $1$ (i.e., $n$ is the entry in the disconnected cell) or
$n-d\in \Des (\jdt(-d+T))$, where $d$ is the entry in the disconnected cell. 
\end{theorem}

\begin{example} Let
\[
T = \young(:::3,124,56).
\]
The entry in the disconnected cell is $d=3$. Computing
\[
-3+T=\young(:::6,451,23) \overset{\ele}{\mapsto} \young(:::6,415,23) \overset{\ele}{\mapsto}
\young(:::6,145,23) \overset{\ele}{\mapsto}
\young(:::6,135,24)=\jdt(-3+T),
\]
we see that $6-3 \in \{1,3\} = \Des(\jdt(-3+T))$, and so $6\in\cDes(T)$.
Therefore, $\cDes(T) =\{3,4,6\}$.
\end{example}

With the definition of $\cDes$ from Theorem~\ref{thm:cDes_lambda_box},
the map $\bij:\SYT(\lambda^\Box)\to\SYT(\lambda^\Box)$ given by
$\bij T=\jdt(1+d+\jdt(-d+T))$ is a bijection that cyclically shifts $\cDes$, as shown in~\cite{ER16}. This bijection generates a $\bbz_n$-action on $\SYT(\lambda^\Box)$.

\begin{example} Below is an orbit of the $\bbz_6$-action generated by $\bij:\SYT((3,2)^\Box)\to\SYT((3,2)^\Box)$, where cyclic descents are marked in boldface \textcolor{red}{\bf red}:
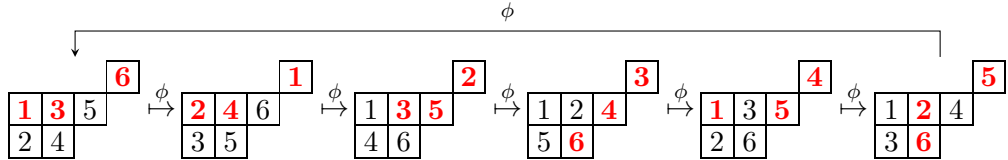
\begin{center}
\begin{tikzpicture}[scale=2.3]
\node (0) at (0,0) {$\young(:::\rsix,\rone\rthree5,24) \overset{\bij}{\mapsto}$};
\node (1) at (1,0) {$\young(:::\rone,\rtwo\rfour6,35) \overset{\bij}{\mapsto}$};
\node (2) at (2,0) {$\young(:::\rtwo,1\rthree\rfive,46) \overset{\bij}{\mapsto}$};
\node (3) at (3,0) {$\young(:::\rthree,12\rfour,5\rsix) \overset{\bij}{\mapsto}$};
\node (4) at (4,0) {$\young(:::\rfour,\rone3\rfive,26) \overset{\bij}{\mapsto}$};
\node (5) at (4.9,0) {$\young(:::\rfive,1\rtwo4,3\rsix)$};
\draw [->, >=stealth] (4.9,.3) -- (4.9,.45) -- node[above]{\scriptsize $\bij$} (-0.1,.45) -- (-0.1,.3); 
\end{tikzpicture}
\end{center}
\end{example}

\section{Near-hooks and the uniqueness of cyclic descents}\label{sec:uniqueness}

\subsection{Uniqueness of cyclic descents}

Recall the following definition from \cite{ARR}.

\begin{definition}[{\cite[Definition 2.1]{ARR}}]\label{def:cDes}
Let $\TTT$ be a finite set. A {\em descent map} is any map
$\Des: \TTT \longrightarrow 2^{[n-1]}$. 
A {\em cyclic extension} of $\Des$ is
a pair $(\cDes,\bij)$, where 
$\cDes: \TTT \longrightarrow 2^{[n]}$ is a map 
and $\bij: \TTT \longrightarrow \TTT$ is a bijection,
satisfying the following axioms:  for all $T$ in  $\TTT$,
\[
\begin{array}{rl}
\text{(extension)}   & \cDes(T) \cap [n-1] = \Des(T),\\
\text{(equivariance)}& \cDes(\bij T)  = 1+\cDes(T),\\
\text{(non-Escher)}  & \varnothing \subsetneq \cDes(T) \subsetneq [n].\\
\end{array}
\]
\end{definition}

By introducing the non-Escher axiom, this definition differs slightly from Definition~\ref{def:extension} above.
However, by~\cite[Theorem 7.8]{ARR}, this makes a difference only for very special shapes, namely those consisting of a single row, a single column, or an even number of single-cell connected components. 

\begin{theorem}[{\cite[Theorem 1.1]{ARR}}]\label{conj1}
Let $\s$ be a skew shape with $n$ cells.
The descent map $\Des$ on $\SYT(\s)$ has a cyclic extension $(\cDes,\bij)$ if and only if $\s$ is not a connected ribbon.
Furthermore, for all $J \subseteq [n]$, all such cyclic extensions share the same cardinalities $\card{\cDes^{-1}(J)}$.
\end{theorem}

Thus, the cardinalities $\card{\cDes^{-1}(J)}$ are unique for all $J \subseteq [n]$ but, in general, the map $\cDes$ itself is not unique. 
In this section we study an interesting family of shapes for which the map $\cDes$ {\em is} unique; note, however, that the bijection $\bij$ is not necessarily unique even in this case.
This is the family of skew shapes for which the descent numbers of all SYT have only two possible (consecutive) values, $k-1$ and $k$. 
Specifically, for a skew shape $\s$ define the set 
\[
D(\s) := \{\card{\Des(T)} \,:\, T \in \SYT(\s)\}.
\]

\begin{lemma}\label{t.unique_cDes}
Let $\s$ be a non-ribbon skew shape with $n$ cells.
If $D(\s) = \{k-1, k\}$ for some $1 \le k \le n-1$, 
then there is a unique cyclic descent map $\cDes: \SYT(\s) \longrightarrow 2^{[n]}$  satisfying 
Definition~\ref{def:extension}, 
defined by
\[
n \in \cDes(T) \iff \card{\Des(T)} = k-1;
\]
equivalently, by requiring $\card{\cDes(T)} = k$ for all $T \in \SYT(\s)$.
\end{lemma}

\begin{proof}
Under the given assumption on $D(\s)$, the size of a cyclic descent set $\cDes(T)$ must be $k-1$, $k$, or $k+1$. 
If $\card{\cDes(T)} = k-1 > 0$ for some $T \in \SYT(\s)$ (note that $\card{\cDes(T)} = 0$ is impossible by the non-Escher axiom), then, by Definition~\ref{def:extension}(ii), there exists $T' \in \SYT(\s)$ with $\card{\cDes(T')} = k-1$ and $n \in \cDes(T')$; but then $\card{\Des(T')} = \card{\cDes(T')} - 1 = k-2$, which is impossible.
Similarly, if $\card{\cDes(T)} = k+1 < n$ for some $T \in \SYT(\s)$ (note that $\card{\cDes(T)} = n$ is impossible by the non-Escher axiom), then there exists $T' \in \SYT(\s)$ with $\card{\cDes(T')} = k+1$ and $n \not\in \cDes(T')$; hence $\card{\Des(T')} = \card{\cDes(T')} = k+1$, which is again impossible.
It follows that, necessarily, $\card{\cDes(T)} = k$ for all $T \in \SYT(\s)$, leading to the claimed uniqueness,
while existence follows from Theorem~\ref{conj1}.
\end{proof}

\begin{remark}
A similar proof shows that if $D(\s) = \{k\}$ consists of a single value, then there is {\em no} cyclic descent map for $\s$. As we shall see below (Theorem~\ref{t.two_descents}), this happens only for a hook or its reverse.
\end{remark}

\subsection{Near-hooks}

We now provide an explicit description of all the shapes satisfying the assumptions of Lemma~\ref{t.unique_cDes}. 

For a SYT $T$, let $T^t$ denote its {\em transpose} (or conjugate) tableau, obtained by reflecting $T$ along the main diagonal, i.e., by interchanging rows with columns.
For $i \in [n-1]$, $i \in \Des(T^t)$ if and only if $i \not\in \Des(T)$.
The {\em reverse} $(\s)^{\rev}$ of a skew shape $\s$ is the skew shape obtained by rotating $\s$ by $180^\circ$. The 
{\em reverse} of a tableau $T \in \SYT(\s)$ is the tableau $T^{\rev} \in \SYT((\s)^{\rev})$ obtained from the $180^\circ$ rotation of $T$ by replacing each entry $i$ by $n+1-i$, where $n$ is the number of cells in $\s$.
It is easy to see that $i \in \Des(T^{\rev})$ if and only if $n-i \in \Des(T)$.

\begin{example}
For $\lambda=(5,4,2,2)$ and $\mu=(2,2,1)$, here is a tableau $T\in\SYT(\s)$ and its corresponding reverse tableau $T^{\rev} \in \SYT((\s)^{\rev})$:
\[
\young(::346,::78,:1,25) \quad\overset{\rev}{\longrightarrow}\quad
\young(:::47,:::8,:12,356).
\]
Their descent sets are $\{1,4,6\}$ and $\{2,4,7\}$, respectively.
\end{example}

For a skew shape $\s$, we write $\s = s_1 \oplus \ldots \oplus s_t$ to indicate that $s_1, \ldots, s_t$ are its connected components, ordered from southwest to northeast. 

\begin{example}
The shape $\s = ((2,2)/(1)) \oplus (1) \oplus (3)$ is
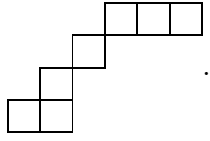
\[
\young(:::\hfill\hfill\hfill,::\hfill,:\hfill,\hfill\hfill).
\]
\end{example}

\medskip

In light of Lemma~\ref{t.unique_cDes}, the main result of this section is the following.

\begin{theorem}\label{t.two_descents}
Let $\s$ be skew shape with $n \ge 2$ cells, and let $1 \le k \le n-1$ be an integer. Then:
\begin{enumerate}
\item
$D(\s) = \{k\}$ if and only if 
$\s$ is either the hook $(n-k,1^k)$ or its reverse.
\item
$D(\s) = \{k-1,k\}$ if and only if 
either $\s$ or its reverse is ``one cell away from a hook'', namely has one of the following forms.
\begin{enumerate}[(a)]
\item
{\em Hook minus its corner cell:} $(n-k+1,1^k)/(1) = (1^k) \oplus (n-k)$.
\item
{\em Hook plus a disconnected cell:} $(n-k,1^{k-1}) \oplus (1)$ or $(1) \oplus (n-k,1^{k-1})$.
\item
{\em Hook plus an internal cell:} $(n-k,2,1^{k-2})$, with $2 \le k \le n-2$.
\end{enumerate}
\end{enumerate}
\end{theorem}

\begin{definition}
The shapes in items (a), (b) and (c) of Theorem~\ref{t.two_descents}(2) will be called {\em near-hooks}.
\end{definition}

\begin{example} 
For $n=5$ and $k=2$, here are the various near-hooks:
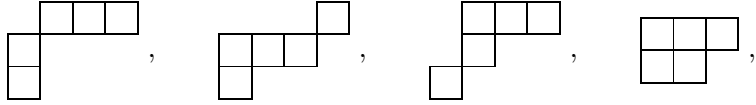
\[
\young(:\hfill\hfill\hfill,\hfill,\hfill)\ , \qquad
\young(:::\hfill,\hfill\hfill\hfill,\hfill)\ ,\qquad
\young(:\hfill\hfill\hfill,:\hfill,\hfill)\ ,\qquad
\young(\hfill\hfill\hfill,\hfill\hfill)\ ,
\]
and here are their reverses:
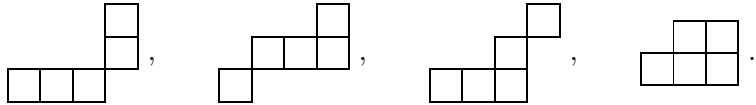
\[
\young(:::\hfill,:::\hfill,\hfill\hfill\hfill)\ , \qquad
\young(:::\hfill,:\hfill\hfill\hfill,\hfill)\ ,\qquad
\young(:::\hfill,::\hfill,\hfill\hfill\hfill)\ ,\qquad
\young(:\hfill\hfill,\hfill\hfill\hfill)\ .
\]
\end{example}

The proof of Theorem~\ref{t.two_descents} will follow from two lemmas. The first one, which is of independent interest, establishes the ``only if'' direction.

\begin{lemma}\label{t.minmax_Des}
Let $\s$ be a skew shape with $n$ cells.
Let $c$ (respectively, $r$) be the maximum length of a column (respectively, row) in $\s$.
Then:
\begin{enumerate}
\item
The minimum and maximum values in $D(\s)$ are 
\[
\min(D(\s)) = c - 1, \qquad
\max(D(\s)) = (n-1) - (r - 1).
\]
In particular, $(c - 1) + (r - 1) \le n - 1$.
\item
If $(c - 1) + (r - 1) = n - 1$ 
(equivalently, $D(\s)$ consists of a single value),
then either $\s$ or its reverse is a hook.
\item
If $(c - 1) + (r - 1) = n - 2$ 
(equivalently, $D(\s)$ consists of two consecutive values),
then either $\s$ or its reverse is a near-hook.
\end{enumerate}
\end{lemma}

\begin{proof}
\begin{enumerate}
\item
Let us first show that, for any $T \in \SYT(\s)$, $\card{\Des(T)} \ge c - 1$. Indeed, choose a column of length $c$ in $\s$, and let $a_1 < \ldots < a_c$ be the entries of $T$ in this column. 
For each $1 \le i \le c-1$, the sequence $a_i, a_i+1, \ldots, a_{i+1}-1$ must contain at least one descent of $T$, since otherwise each of these numbers would appear in a cell strictly west and weakly south of its successor, forcing $a_i$ to be strictly west and weakly south of $a_{i+1}$, which is impossible. We have thus established the existence of $c-1$ distinct elements in $\Des(T)$.

Let us now construct a tableau $T \in \SYT(\s)$ for which $\card{\Des(T)} = c - 1$. As a guiding example, if $\s$ is a straight shape (i.e., if $\mu$ is empty), fill the shape row by row, from top to bottom, with each row filled consecutively from left to right. Note that $c$ is then the number of rows in $\lambda$. Similarly, for a general (possibly disconnected) skew shape $\s$, define {\it skew rows} as follows. The first skew row consists of the top cells of all the columns in $\s$.
Removing these cells leaves another skew shape, for which we iterate the same procedure to define the second skew row, and so on. Now fill these skew rows from top to bottom, with each skew row filled consecutively from left to right. Since the length of each column decreases by $1$ at each step (as long as it is still nonzero), the number of skew rows is exactly $c$. The descents of the resulting tableau $T$ are the entries in the last cell of each skew row (except the last). Thus $\card{\Des(T)} = c - 1$ and $\min(D(\s)) = c - 1$ as claimed.

The argument for $\max(D(\s))$ is similar; alternatively, it follows by considering the conjugate shape, and noting that $\Des(T^t) = [n-1] \setminus \Des(T)$.
Finally, the inequality $(c - 1) + (r - 1) \le n - 1$ follows from the inequality $\min(D(\s)) \le \max(D(\s))$. 

\item
If $(c-1)+(r-1) = n-1$, consider a column of length $c$ and a row of length $r$ in $\s$. Since $c+r = n+1$, they must have a (unique) cell in common, and their union contains all the cells of $\s$.
If the positive integers $i < i'$ and $j < j'$ are such that the cells $(i,j)$ and $(i',j')$ belong to the skew shape $\s$, then so do $(i,j')$ and $(i',j)$. We deduce that the common cell is either the top cell of the column as well as the leftmost cell of the row, or it is the bottom cell of the column as well as the rightmost cell of the row. Thus either $\s$ or its reverse is a hook.

\item
If $(c-1)+(r-1) = n-2$, consider a column of length $c$ and a row of length $r$. If they do not intersect, then their union contains all the cells in $\s$, and they must belong to distinct connected components of $\s$. It follows that either $\s$ or its reverse is a hook minus its corner cell (when the components are taken to be as close together as possible).
If the column and row do intersect (necessarily in a single cell), then their union has size $n-1$, and there is a unique cell in $\s$ outside the union. If this cell belongs to a different connected component, then the other component is a hook or its reverse, as before. If it belongs to the same connected component, then a skew shape argument as in part 2 above shows that either $\s$ or its reverse is a hook plus an internal cell.\qedhere
\end{enumerate}
\end{proof}

To prove the ``if'' direction of Theorem~\ref{t.two_descents}, we need a detailed description of the descent set in each case. For a set $C$, recall the notation $C-1=\{c-1: c\in C\}$.

\begin{lemma}\label{t.near-hooks}
Assume that either $\s$ or its reverse is a near-hook with $n$ cells, and
let $k$ be as given in each case of Theorem~\ref{t.two_descents}.2. Then
$$D(\s) = \{k-1,k\}.$$ Additionally, for $T\in \SYT(\s)$, $\Des(T)$ is given explicitly in each case as follows.

\begin{enumerate}[(a)]
\item
{\em Hook minus its corner cell:} Let $C$ be the set of entries in the column $(1^k)$ of $T$.
$$\begin{array}{c|c|c}
\s & \Des(T) & \card{\Des(T)}=k-1 \text{ iff} \\ \hline \hline
(1^k) \oplus (n-k) & (C \setminus \{1\})-1 & 1 \in C \\ \hline
((1^k) \oplus (n-k))^{\rev} & C \setminus \{n\} & n \in C
\end{array}$$

\item 
{\em Hook plus a disconnected cell:}  
Let $T_0$ be the entry in the disconnected cell of $T$, and let $C$ be the set of entries in the column $(1^{k-1})$ (excluding the corner cell).
$$\begin{array}{c|c|c}
\s & \Des(T) & \card{\Des(T)}=k-1 \text{ iff} \\ \hline \hline

(n-k,1^{k-1}) \oplus (1) & \begin{cases}
C-1& \text{if } T_0 \in (C-1) \sqcup \{n\}, \\
(C-1) \sqcup \{T_0\}& \text{otherwise}
\end{cases} & T_0 \in (C-1) \sqcup \{n\} \\ \hline

((n-k,1^{k-1}) \oplus (1))^{\rev} & \begin{cases}
C& \text{if } T_0 \in (C+1) \sqcup \{1\}, \\
C \sqcup \{T_0-1\}& \text{otherwise}
\end{cases} & T_0 \in (C+1) \sqcup \{1\}  \\ \hline

(1) \oplus (n-k,1^{k-1}) & \begin{cases}
C-1& \text{if } T_0 = 1, \\
((C-1) \setminus \{T_0\}) \sqcup \{T_0-1\}& \text{if } T_0 \in C-1, \\
(C-1) \sqcup \{T_0-1\}& \text{otherwise}
\end{cases} & T_0 \in (C-1) \sqcup \{1\} \\ \hline

((1) \oplus (n-k,1^{k-1}))^{\rev} & \begin{cases}
C& \text{if } T_0 = n, \\
(C \setminus \{T_0-1\}) \sqcup \{T_0\}& \text{if } T_0 \in C+1, \\
C \sqcup \{T_0\}& \text{otherwise}
\end{cases} & T_0 \in (C+1) \sqcup \{n\}
\end{array}$$

\item
{\em Hook plus an internal cell:} 
Let $T_{\inte}=T_{2,2}$ be the entry in the internal cell of $T$, and let $C$ be the set of entries in the column $(1^{k-1})$ (excluding the corner cell).
$$\begin{array}{c|c|c}
\s & \Des(T) & \card{\Des(T)}=k-1 \text{ iff} \\ \hline \hline
(n-k,2,1^{k-2}) & \begin{cases}
C-1& \text{if } T_{\inte} \in C+1, \\
(C-1) \sqcup \{T_{\inte}-1\}& \text{otherwise}
\end{cases} & T_{\inte} \in C+1 \\ \hline
(n-k,2,1^{k-2})^{\rev} & \begin{cases}
C& \text{if } T_{\inte} \in C-1, \\
C \sqcup \{T_{\inte}\}& \text{otherwise}
\end{cases} & T_{\inte} \in C-1
\end{array}$$

\end{enumerate}
\end{lemma}

\begin{proof}
The equality $D(\s) = \{k-1,k\}$ follows from the detailed description of $\Des(T)$ in each case, which can be proved 
by a straightforward case-by-case verification. Details are omitted, but
we remark that if $\s = (1) \oplus (n-k,1^{k-1})$, it is impossible to have $T_0 = 1 \in C-1$; and if $\s = ((1) \oplus (n-k,1^{k-1}))^{\rev}$, it is impossible to have $T_0 = n \in C+1$.
\end{proof}

\begin{proof}[Proof of Theorem~\ref{t.two_descents}]
\begin{enumerate}
\item
The ``only if'' direction follows from Lemma~\ref{t.minmax_Des}.2.
The ``if'' direction is clear by noting that $D(\s) = \{k\}$ if either $\s$ or its reverse is the hook $(n-k,1^k)$.
\item
The ``only if'' direction follows from Lemma~\ref{t.minmax_Des}.3,
and the ``if'' direction is given by Lemma~\ref{t.near-hooks}.\qedhere
\end{enumerate}
\end{proof}

Having shown that $D(\s)=\{k-1,k\}$ for some $1\le k\le n-1$ when $\s$ or its reverse is a near-hook, Lemma~\ref{t.unique_cDes} guarantees that there is a unique cylic descent map for such shapes. The explicit conditions determining whether $n \in \cDes(T)$ for $T$ of such shapes can now be extracted from Lemma~\ref{t.near-hooks}.

\begin{corollary}\label{t.n_in_cDes}
Assume that either $\s$ or its reverse is a near-hook with $n$ cells, 
let $k$ be as given in each case of Theorem~\ref{t.two_descents}.2, and
let $T\in\SYT(\s)$. Then:
\begin{enumerate}
\item
There exists a unique cyclic descent map for $\s$, specified by letting
\[
n \in \cDes(T) \iff \card{\Des(T)} = k-1.
\]
In particular, $\card{\cDes(T)}=k$.

\item
Explicitly, the condition for $n\in\cDes(T)$ is given in each case as follows, where $C$, $T_0$ and $T_{\inte}$ are defined as in Lemma~\ref{t.near-hooks}.
\[
\begin{array}{l|c|c}
& \s & n\in \cDes(T) \text{ iff} \\ \hline \hline
\text{(a) {\em Hook minus its corner cell:}} & (1^k) \oplus (n-k) & 1 \in C \\ \cline{2-3}
& ((1^k) \oplus (n-k))^{\rev} & n \in C \\ \hline
\text{(b) {\em Hook  plus a disconnected cell:}} & (n-k,1^{k-1}) \oplus (1) & T_0 \in (C-1) \sqcup \{n\} \\ \cline{2-3}
& ((n-k,1^{k-1}) \oplus (1))^{\rev} & T_0 \in (C+1) \sqcup \{1\} \\ \cline{2-3}
& (1) \oplus (n-k,1^{k-1}) & T_0 \in (C-1) \sqcup \{1\} \\ \cline{2-3}
& ((1) \oplus (n-k,1^{k-1}))^{\rev} & T_0 \in (C+1) \sqcup \{n\} \\ \hline
\text{(c) {\em Hook plus an internal cell:}} & (n-k,2,1^{k-2}) & T_{\inte} \in C+1 \\ \cline{2-3}
& (n-k,2,1^{k-2})^{\rev} & T_{\inte} \in C-1 
\end{array}
\]
\end{enumerate}
\end{corollary}

Case (a) will be generalized in Section~\ref{sec:strips};
case (c) is Theorem~\ref{thm:hooks_plus1_intro} above. 

\smallskip

Two interesting symmetries for cyclic descents of near-hooks emerge from Lemma~\ref{t.near-hooks} and Corollary~\ref{t.n_in_cDes}. 
They correspond to symmetries for descents of arbitrary skew shapes.

\begin{corollary}\label{t.transpose_reverse}
Let $\s$ be a skew shape with $n$ cells and let $T \in \SYT(\s)$. Then
\[
\Des(T^t) = [n-1] \setminus \Des(T),
\qquad
\Des(T^{\rev}) = n - \Des(T).
\]
If, additionally, either $\s$ or its reverse is a near-hook, then
the (unique) cyclic descent set of the transpose and the reverse tableaux are
\[
\cDes(T^t) = [n] \setminus \cDes(T),
\qquad
\cDes(T^{\rev}) = n - \cDes(T),
\]
where $0$ is interpreted as $n \pmod n$.
\end{corollary}

Finally, we compute the multiplicities $\card{\cDes^{-1}(J)}$ for each $J \subseteq [n]$ with $|J| = k$, for the various types of near-hooks.
A {\em cyclic run} of a set $\varnothing \subsetneq J \subsetneq [n]$ is a maximal cyclic interval $\{i,i+1,\ldots,i+r\} \pmod n$ contained in $J$.

\begin{theorem}\label{t.near-hook:fibers}
Assume that either $\s$ or its reverse is a near-hook with $n$ cells, and let $k$ be such that $\card{\cDes(T)} = k$ for every $T \in \SYT(\s)$ (which exists by Corollary~\ref{t.n_in_cDes}.1). For every $J \subseteq [n]$ with $|J| = k$, 
$\card{\cDes^{-1}(J)}$ is given in each case as follows.
\begin{enumerate}[(a)]
\item
{\em Hook minus its corner cell:}
If $\s$ is either $(1^k) \oplus (n-k)$ or its reverse, then
$\card{\cDes^{-1}(J)} = 1$.

\item
{\em Hook plus a disconnected cell:}
If $\s$ is one of $(n-k,1^{k-1}) \oplus (1)$, $(1) \oplus (n-k,1^{k-1})$, or their reverses, then
$\card{\cDes^{-1}(J)}$ equals the number of cyclic runs in $J$.

\item
{\em Hook plus an internal cell:}
If $\s$ is either $(n-k,2,1^{k-2})$ or its reverse,
then $\card{\cDes^{-1}(J)}$ equals the number of cyclic runs in $J$ minus $1$.
\end{enumerate}
\end{theorem}

\begin{proof}
By Corollary~\ref{t.transpose_reverse}, reversing the shape if needed, we can assume that $(\s)^{\rev}$ is a near-hook. Similarly, transposing the shape if needed, we can reduce case (b) to one shape. Let $T \in \SYT(\s)$.
\begin{enumerate}[(a)]
\item
If $\s = ((1^k) \oplus (n-k))^{\rev}$, then, by Lemma~\ref{t.near-hooks}, $\cDes(T)$ is the set of entries in the column $(1^k)$ of $T$. 
Thus, $\cDes(T)$ uniquely determines $T$.

\item
If $\s = ((n-k,1^{k-1}) \oplus (1))^{\rev}$, then, by Lemma~\ref{t.near-hooks} and Corollary~\ref{t.n_in_cDes},
\[
\cDes(T) = 
\begin{cases}
C \sqcup \{n\}& \text{if } T_0 \in (C+1) \sqcup \{1\}, \\
C \sqcup \{T_0-1\}& \text{otherwise,}
\end{cases}
\]
where $C$ is the set of entries in the column $(1^{k-1})$ of $T$, and $T_0$ is the entry in the disconnected cell.
Since $\cDes$ satisfies Definition~\ref{def:extension}(ii), it suffices to prove the claim for $J \subseteq [n-1]$ (i.e., not containing $n$). 
In this case, $J = C \sqcup \{T_0-1\}$ and $T_0 \not\in (C+1) \sqcup \{1\}$.
Clearly, $C \subseteq [n-1]$ (of size $k-1$) and $T_0-1 \in [n-1] \setminus C$ determine a unique $T \in \SYT(\s)$, provided that $T_0 \not\in C$.
Note that either $T_0 \ne n$ (so that $n$ occupies the bottom-right corner) or $T_0 = n$ (so that $n-1 = T_0-1$ occupies that corner).
Summing up, given $J \subseteq [n-1]$, to construct $T\in\SYT(\s)$ with $\cDes(T)=J$ one can choose $T_0-1 \in J$ subject only to $T_0 \not\in J$, yielding 
\[
\card{\cDes^{-1}(J)}
= \card{\{j \in J \,:\, j+1 \not\in J\}},
\]
which is the number of cyclic runs in $J$.

\item
If $\s = (n-k,2,1^{k-2})^{\rev}$, then, by Lemma~\ref{t.near-hooks} and Corollary~\ref{t.n_in_cDes},
\[
\cDes(T) = 
\begin{cases}
C \sqcup \{n\},& \text{if } T_{\inte} \in C-1; \\
C \sqcup \{T_{\inte}\},& \text{otherwise,}
\end{cases}
\]
where $C$ is the set of entries in the column $(1^{k-1})$ of $T$, and $T_{\inte}$ is the entry in the internal cell.
By Definition~\ref{def:extension}(ii), it suffices to prove the claim for $J \subseteq [n-1]$. 
In this case, $J = C \sqcup \{T_{\inte}\}$ and $T_{\inte} \not\in C-1$.
Clearly $C \subseteq [n-1]$ (of size $k-1$) and $T_{\inte}+1 \in [n] \setminus C$ determine a unique $T \in \SYT(\s)$, provided that $T_{\inte} \not\in C$, $T_{\inte} < \max(C)$ and $T_{\inte} < \max([n-1] \setminus C)$ (given that $n$ must occupy the bottom-right corner).
In fact, the last equality is superfluous, since $T_{\inte} < \max(C)$ implies $T_{\inte} \le n-2$, and the constraint $T_{\inte}+1 \not\in C$ then implies $T_{\inte} < T_{\inte}+1 \le \max([n-1] \setminus C)$.
Summing up, given $J \subseteq [n-1]$, we choose $T_{\inte} \in J$ such that $T_{\inte}+1 \not\in J$ and $T_{\inte} < \max(J)$, yielding 
\[
\card{\cDes^{-1}(J)}
= \card{\{j \in J \,:\, j+1 \not\in J,\, j < \max(J)\}}
= \card{\{j \in J \,:\, j+1 \not\in J\}} - 1,
\]
which is one less than the number of cyclic runs in $J$.\qedhere
\end{enumerate}
\end{proof}

\subsection{The bijection $\mathbf{\bij}$ for near-hooks}
\label{sec:hooks_plus_action} 

The inverse promotion operator $p^{-1}T=\jdt(1+T)$
shifts the cyclic descent set for rectangular shapes (see Section~\ref{sec:prel2}) and
for hooks minus their corner cell (see the more general discussion of strips in Section~\ref{sec:strips} below). Namely, for every SYT $T$ 
of rectangular or strip shape, $\cDes(p^{-1}T)=1+\cDes(T)$, with addition modulo $n$. However, for hooks plus an internal or disconnected cell, this is not always the case.
For example,
\[
\cDes\left(\young(125,34)\right)=\{2,5\},\ \ \ \ \text{but}\ \ \ \ 
\cDes\left(p^{-1}\young(125,34)\right)=\cDes\left(\young(123,45)\right)=\{3,5\}.
\]
A variant of promotion which shifts the cyclic descent set for hooks plus a disconnected cell was described in~\cite{ER16}, see Section~\ref{sec:prel2}.
In this section we present a bijection on SYT of a shape consisting of a hook plus an internal cell, cyclically shifting their (uniquely defined) cyclic descent sets.
Instead of describing a bijection $\bij$ such that $\cDes(\bij T)=1+\cDes(T)$, it will be more convenient to describe its inverse $\psi=\bij^{-1}$, satisfying $\cDes(\psi T)=-1+\cDes(T)$, always with addition modulo $n$, the number of cells.
We emphasize that, even though $\cDes$ is uniquely defined for these shapes, the definition of $\bij$ (or $\psi$) is not unique.

We assume throughout this section that $2\le k\le n-2$. 
A {\em run} of a set $J \subseteq [n-1]$ is a maximal interval $\{i,i+1,\ldots,i+r\}$ contained in $J$. 
In this case, $i$ is called the first letter of the run.
The {\em first run} of $J$ is the one that contains its smallest element.
Note that, since $n \not\in J$, runs are the same as the cyclic runs (defined before Theorem~\ref{t.near-hook:fibers}) when viewing $J$ as a subset of $[n]$.

\begin{lemma}\label{lem:hook2}
Let $T\in \SYT(n-k,2,1^{k-2})$, and let $\cDes$ be defined as in Theorem~\ref{thm:hooks_plus1_intro} (equivalently, Corollary~\ref{t.n_in_cDes}.2(c)).
\begin{itemize}
\item
If	$n\not\in \cDes(T)$, 
then $T_{2,2}-1$ is the first letter in a run of $\cDes(T)$ which is not the first run.

\item
If $n\in \cDes(T)$, then $T_{2,2}-1$ is the first letter in a run of $\cDes(T^t)$ which is not the first run.
\end{itemize}
\end{lemma}

\begin{proof}
To prove the first part, note that, by Lemma~\ref{t.n_in_cDes}.2(c),
$n \not\in \cDes(T)$ if and only if $T_{2,2}-1$ is in the first row of $T$.
Thus, in this case, $T_{2,2}-1 \in \Des(T)$ but $T_{2,2}-2\notin \Des(T)$. 
It follows that $T_{2,2}-1$ is the first letter in a run of $\Des(T)=\cDes(T)$, which is not the first run because $T_{2,1}-1\in\Des(T)$.

The second part follows using Corollary~\ref{t.transpose_reverse} and noting that $T^t_{2,2}=T_{2,2}$.
\end{proof}

\begin{definition}\label{def:psi}
For every $2\le k\le n-2$, define a map $\psi:\SYT(n-k,2,1^{k-2}) \to \SYT(n-k,2,1^{k-2})$
as follows.
\begin{itemize}
\item 
If $1 \not\in \cDes(T)$, let the first row of $\psi T$ contain the entries $[n] \setminus\cDes(T)$,
in increasing order from left to right. 
By Lemma~\ref{lem:hook2}, $T_{2,2}-1$ is the first letter in the $i$th run (for some $i \ne 1$) of either $\cDes(T)$ or $\cDes(T^t)$. In any case, let $(\psi T)_{2,2}$ be the first letter in the $i$th run of $\cDes(T)$.
Equivalently,
\[
(\psi T)_{2,2} =
\begin{cases}
T_{2,2}-1& \text{if $T_{2,2}-1$ is in the first row of $T$ (i.e., if $n\notin\cDes(T)$),} \\
a-1 & \text{if $T_{2,2}-1$ is in the first column of $T$ and $a$ is the entry below it,} \\
n & \text{if $T_{2,2}-1$ is at the bottom of the first column of $T$.} \end{cases}
\]
Place the remaining letters in the first column of $\psi T$, in increasing order from top to bottom.

\item 
If $1 \in \cDes(T)$, define $\psi T := (\psi T^t)^t$, reducing to the previous case.   
    
\end{itemize}
\end{definition}

\begin{remark}
For $T\in\SYT(n-k,2,1^{k-2})$ with $1,n$ both in $\cDes(T)$ or 
both not in $\cDes(T)$, the map $\psi$ coincides with the promotion operator.
\end{remark}

\begin{proposition}
The map $\psi$ is a bijection (generating a $\ZZ_n$-action) on $\SYT(n-k,2,1^{k-2})$ which cyclically shifts the cyclic descent set $\cDes$ defined in Theorem~\ref{thm:hooks_plus1_intro} (or Corollary~\ref{t.n_in_cDes}.2(c)); 
namely, for every $T\in \SYT(n-k,2,1^{k-2})$, $\cDes(\psi T)=-1+ \cDes(T)$ with addition modulo~$n$.
\end{proposition}

\begin{example}
Below are three orbits of the $\bbz_6$-action on $\SYT(3,2,1)$. Cyclic descents are marked in boldface \textcolor{red}{\bf red}.

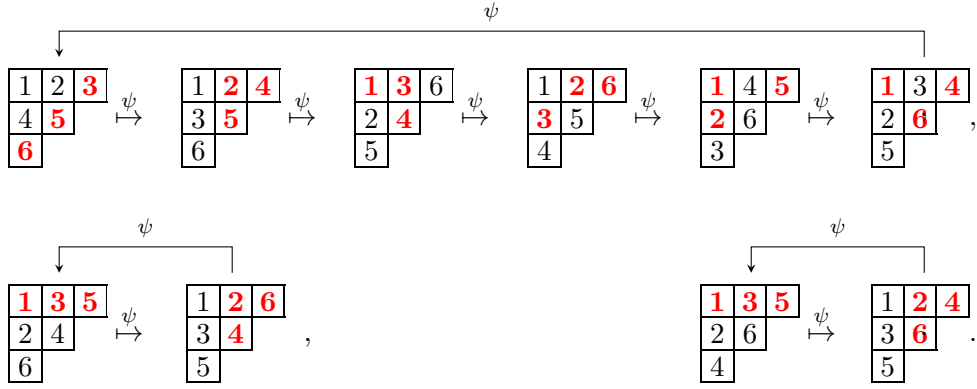
\begin{center}
\begin{tikzpicture}[scale=2.3]
\node (0) at (0,0) {$\young(12\rthree,4\rfive,\rsix)\overset{\psi}{\mapsto}$};
\node (1) at (1,0) {$\young(1\rtwo\rfour,3\rfive,6)\overset{\psi}{\mapsto}$};
\node (2) at (2,0) {$\young(\rone\rthree6,2\rfour,5)\overset{\psi}{\mapsto}$};
\node (3) at (3,0) {$\young(1\rtwo\rsix,\rthree5,4)\overset{\psi}{\mapsto}$};
\node (4) at (4,0) {$\young(\rone4\rfive,\rtwo6,3)\overset{\psi}{\mapsto}$};
\node (5) at (4.9,0) {$\young(\rone3\rfour,2\rsix,5)$,};
\node (5) at (4.9,0) {$$,};
\draw [->, >=stealth] (4.9,.35) -- (4.9,.5) -- node[above]{\scriptsize $\psi$}  (-0.1,.5) -- (-0.1,.35);  
\end{tikzpicture}
\end{center}

\begin{center}
\begin{tikzpicture}[scale=2.3]
\node (0) at (0,0) {$\young(\rone\rthree\rfive,24,6)\overset{\psi}{\mapsto}$};
\node (1) at (1,0) {$\young(1\rtwo\rsix,3\rfour,5)$\ \ ,};
\draw [->, >=stealth] (0.9,.35) -- (0.9,.5) -- node[above]{\scriptsize $\psi$}  (-0.1,.5) -- (-0.1,0.35);  
\node (4) at (4,0) {$\young(\rone\rthree\rfive,26,4)\overset{\psi}{\mapsto}$};
\node (5) at (4.9,0) {$\young(1\rtwo\rfour,3\rsix,5)$.};
\node (5) at (4.9,0) {$$.};
\draw [->, >=stealth] (4.9,.35) -- (4.9,.5) -- node[above]{\scriptsize $\psi$}  (3.9,.5) -- (3.9,0.35);  
\end{tikzpicture}
\end{center}
\end{example}

\begin{proof} 
First we show that $\psi T$ is a standard Young tableau. 
By definition, the first row and first column are increasing, so it remains to show that 
 $(\psi T)_{2,2}>\max \{(\psi T)_{1,2}, (\psi T)_{2,1}\}$.
By Corollary~\ref{t.transpose_reverse} (regarding the effect of the transpose on $\cDes$), 
it suffices to prove this inequality for the case $1 \not\in \cDes(T)$.
In this case, by Definition~\ref{def:psi} and Lemma~\ref{lem:hook2}, $(\psi T)_{2,2}$ belongs to a run of $\cDes(T)$ which is not the first run,
    thus it is bigger than the minimum of $\cDes(T)$, which equals $(\psi T)_{2,1}$.
Also, by definition, if $1\not\in \cDes(T)$ then 
$(\psi T)_{1,2}$ is the smallest value in $[n]\setminus\cDes(T)$ which is not 1,
thus smaller than first letter in any run of $\cDes(T)$ which is not
the first run, hence $(\psi T)_{1,2}\le (\psi T)_{2,2}$.

\medskip

Next we prove that $\cDes(\psi  T)=-1+\cDes(T)$ . Again, by Corollary~\ref{t.transpose_reverse}, 
it suffices to prove this statement for the case $1\not\in \cDes(T)$.
Since $(\psi T)_{2,2}$ is the first letter in a run of $\cDes(T)=\{(\psi T)_{i,j}:\ i>1\}$ which is not the first run, $(\psi T)_{2,2}-1$ lies in the first row of $\psi T$. 
Now observe that, by Lemma~\ref{t.near-hooks}(b), for every $T\in\SYT(n-k,2,1^{k-2})$,
\[
\cDes(T)=\begin{cases}
 \{T_{i,1}-1:\ i>1\}\sqcup \{T_{2,2}-1\} &  \text{if $T_{2,2}-1$ lies in the first row of $T$},\\
 \{T_{i,1}-1:\ i>1\}\sqcup \{n\} & \text{otherwise}.
 \end{cases}
\]
Hence,
\begin{equation}\label{eq:psi}
\cDes(\psi T)=\{(\psi T)_{i,1}-1:\ i>1\}\sqcup\{(\psi T)_{2,2}-1\}=\{(\psi T)_{i,j}-1:\ i>1\}=-1+\cDes(T).
\end{equation}
The last equality follows from the assumption $1\not\in \cDes(T)$ and the definition of $\psi$.

\medskip

To prove that  $\psi$ is a bijection, it suffices to show that $\psi^n$ is the identity map.
For every $\varnothing\ne J\subseteq [n]$, let
$r(J):=|\{d\in J:\ d-1\not\in J \}|$ (with addition modulo $n$) be the number of cyclic runs in $J$.
Since $\cDes(T^t)=[n]\setminus \cDes(T)$ for $T\in\SYT(n-k,2,1^{k-2})$ by Corollary~\ref{t.transpose_reverse},
$r(\cDes(T^t))=r(\cDes(T))$.
If $1\not\in \cDes(T)$, then
all cyclic runs in $\cDes(T)$ are regular runs. Similarly, if $1\in \cDes(T)$, then by  Corollary~\ref{t.transpose_reverse},
$1\not\in \cDes(T^t)$ and all cyclic runs in $\cDes(T^t)$ are regular runs.
Thus, by Lemma~\ref{lem:hook2}, for every $T\in\SYT(n-k,2,1^{k-2})$ there exists a unique 
$1<i\le r(\cDes(T))$ such that
$T_{2,2}-1$ is either the first letter in the $i$th run of $\cDes(T)$ 
or the first letter in the $i$th run of $\cDes(T^t)$. 
Let $f$ be the map from the set $\SYT(n-k,2,1^{k-2})$ to the interval $[2,r(\cDes(T))]$
which associates, to each $T\in\SYT(n-k,2,1^{k-2})$, this positive integer $i$.

Combining this with Theorem~\ref{t.near-hook:fibers}(c), we deduce that
the {\em marked} cyclic descent set map $\cDes\times f$ from $\SYT(n-k,2,1^{k-2})$ to the set of pairs
\[
\{(J,i):\ J\subseteq[n],\ |J|=k+2 \text{ and } 2 \le i \le r(J)\}
\]
is a bijection. 
Furthermore,
by definition, $f$ is invariant under $\psi$, namely, $f(\psi T)=f(T)$ for every $T\in\SYT(n-k,2,1^{k-2})$.
Combining this with Equation \eqref{eq:psi}, the 
length of the orbit of $T$ under $\psi$ is equal to the length of the orbit of 
$\cDes(T)$ under
the shift operation on subsets of $[n]$, thus it divides $n$. One deduces that $\psi^n$ is the identity map, and so $\psi$ is a bijection that generates a $\ZZ_n$-action on $\SYT(n-k,2,1^{k-2})$.
\end{proof}

\subsection{Generating functions for cyclic descent set over near-hook shapes}\label{sec:GF-near-hooks}

Next we derive some enumerative results, for the various types of near-hook shapes, from the explicit descriptions of cyclic descent sets given in Theorem~\ref{t.near-hook:fibers}.
For a subset $J\subseteq [n]$, let ${\bf x}^J:=\prod\limits_{i\in J}x_i$.

\begin{proposition}\label{prop:gf1} We have
\begin{equation}\label{eq:gf1}
\sum_{k=1}^{n-1} \sum_{T \in \SYT ((1^{n-k}) \oplus (k))} {\bf x}^{\cDes(T)}
= \prod_{i=1}^n (1+x_i) - 1 - x_1 \cdots x_n.
\end{equation}
In particular,
\[
\sum_{k=1}^{n-1} \sum_{T \in \SYT ((1^{n-k}) \oplus (k))} q^{\cdes(T)}
= (1+q)^n - 1 - q^n.
\]
\end{proposition}

\begin{proof}
By Theorem~\ref{t.near-hook:fibers}(a), each subset of $[n]$ of size $k$ appears exactly once as $\cDes(T)$ for some $T \in \SYT((1^{n-k})\oplus (k))$.
\end{proof}

\begin{proposition}\label{prop:gf2} We have
\begin{equation}\label{eq:gf2}
\sum_{k=1}^{n-1} \sum_{T \in \SYT((n-k,1^{k-1})\oplus (1))} {\bf x}^{\cDes(T)}
= \prod_{i=1}^{n} (1+x_i) \cdot \sum_{j=1}^{n}\frac{x_j}{(1+x_{j-1})(1+x_j)}.
\end{equation}
In particular,
\[
\sum_{k=1}^{n-1}\sum_{T \in \SYT((n-k,1^{k-1}) \oplus (1))} q^{\cdes(T)}
= nq (1+q)^{n-2}.
\]
\end{proposition}

\begin{proof}
By Theorem~\ref{t.near-hook:fibers}(b), the multiplicity of each monomial ${\bf x}^J$ in the LHS of Equation~\eqref{eq:gf2} (for $\varnothing \subsetneq J \subsetneq [n]$) is equal to the number of cyclic runs in $J$. Representing each cyclic run by its first element, this is the number of $j \in [n]$ for which $j \in J$ but $j-1 \not\in J$ (with subtraction modulo $n$). Changing the order of summation, this amounts to counting, for each $j \in [n]$, each subset $J$ which contains $j$ but not $j-1$.
\end{proof}

\begin{proposition}\label{prop:gf3}
We have
\begin{equation}\label{eq:gf3}
\sum_{k=2}^{n-2} \sum_{T \in \SYT((n-k,2,1^{k-2}))} {\bf x}^{\cDes(T)}
= 1 + x_1 \cdots x_n + \prod_{i=1}^{n} (1+x_i) \cdot \left[ -1 + \sum_{j=1}^{n}\frac{x_j}{(1+x_{j-1})(1+x_j)} \right].
\end{equation}
In particular,
\[
\sum_{k=2}^{n-2} \sum_{T \in \SYT((n-k,2,1^{k-2}))} q^{\cdes(T)}
= nq (1+q)^{n-2} - (1+q)^n+1+q^n.
\]
\end{proposition}

\begin{proof}
By Theorem~\ref{t.near-hook:fibers}(c), the multiplicity of each monomial ${\bf x}^J$ in the LHS of Equation~\eqref{eq:gf3} (for $\varnothing \subsetneq J \subsetneq [n]$) is equal to the number of cyclic runs in $J$ minus $1$.
Therefore, this equation is obtained by subtracting Equation~\eqref{eq:gf1} from Equation~\eqref{eq:gf2}.
\end{proof}

\section{Strips}\label{sec:strips}

\subsection{From near-hooks to strips}\label{sec:near-hooks_to_strips}

In this section we generalize the cyclic descent set map on a hook minus a corner cell to certain shapes called strips. 

A {\em strip} is a skew shape all of whose connected components have either only one row or only one column; equivalently, it is a skew shape which
contains neither of the shapes
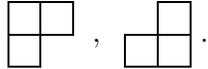
\[
\young(\hfill\hfill,\hfill) \ \ ,\ \
\young(:\hfill,\hfill\hfill)\ .
\]

\begin{example} The following shapes are strips:
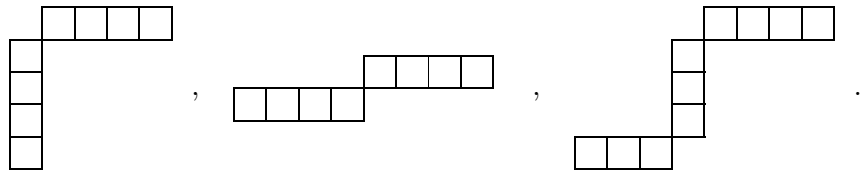
\[
\young(::::::::\hfill\hfill\hfill\hfill,:::::::\hfill,:::::::\hfill,:::::::\hfill,:::::::\hfill)\ \ ,
\quad 
\young(::::\hfill\hfill\hfill\hfill,\hfill\hfill\hfill\hfill)\ \ \ \ ,
\quad 
\young(::::\hfill\hfill\hfill\hfill,:::\hfill,:::\hfill,:::\hfill,\hfill\hfill\hfill)\ \   . 
\]
\end{example}

\medskip

By Corollary~\ref{t.n_in_cDes}.2(a), there exists a unique cyclic descent map for SYT of shape a hook minus its corner,
and it is defined by letting
$n \in \cDes(T)$ if and only if the entry $1$ is in the first column of $T$. 
The existence part of this statement
can be generalized to arbitrary strips. The uniqueness part does not generalize, as the example in Remark~\ref{rem:skew}.2 shows.

\begin{proposition}\label{def-strip}
For every strip $\lambda/\mu$ with $n$ cells, there exists a cyclic descent map, specified for $T\in\SYT(\lambda/\mu)$ by letting $n\in\cDes(T)$ if and only if 
either $n$ is strictly north of $1$, or $1$ and $n$ are
in the same vertical component; equivalently,
if $n-1 \in \Des (pT)$, where $p$ is the promotion operator. 
\end{proposition}

\begin{proof}
For $T\in\SYT(\lambda/\mu)$, $pT$ is obtained from $-1+T$ by sorting the entries within each row so that they
increase from left to right, and within each column so that they increase from top to bottom.
Since $p: \SYT(\lambda/\mu) \to \SYT(\lambda/\mu)$ is a bijection (and, in fact, generates a $\ZZ_n$-action in this case), it is enough to
show that $\cDes(pT)=-1+\cDes(T)$ with the above definition, by considering
the following cases, for $1\le i\le n$, with addition mod $n$:
\begin{itemize}
\item 
If $i$ and $i+1$ are in different connected components of $T$, then the corresponding components of $p T$ contain $i-1$ and $i$, respectively.
Thus, $i\in\cDes(T)$ if and only if $i-1\in\cDes(pT)$.
\item 
If $i$ and $i+1$ are in the same row of $T$, then $i-1$ and $i$ are in the corresponding row of $p T$. 
In this case,  $i\not\in\cDes(T)$ and $i-1\not\in\cDes(p T)$.
\item 
If $i$ and $i+1$ are in the same column of $T$, then $i-1$ and $i$ are in the corresponding column of $p T$.
In this case, $i\in\cDes(T)$ and $i-1\in\cDes(p T)$.\qedhere
\end{itemize}
\end{proof}

We conclude that the inverse promotion operator $\bij T:=\jdt(1+T)=p^{-1} T$ shifts (cyclically) the cyclic descent sets of the SYT of any given strip shape.

\begin{corollary}\label{cor:strips_action}
For every SYT $T$ of a strip shape, 
\[
\cDes(\bij T)=1+\cDes(T).
\]
\end{corollary}

\begin{example} 
Iteratively applying $\bij$ to the SYT below, we get
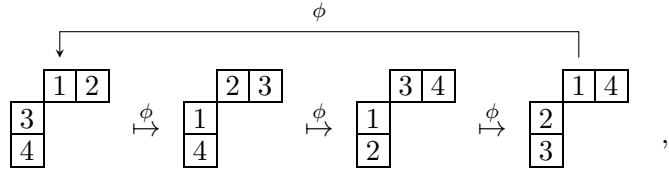
\begin{center}
\begin{tikzpicture}[scale=2.3]
\node (0) at (0,0) {$\young(:12,3,4)$};
\node (1) at (0.5,0) {$\overset{\bij}{\mapsto}$};
\node (2) at (1,0) {$\young(:23,1,4)$};
\node (3) at (1.5,0) {$\overset{\bij}{\mapsto}$};
\node (4) at (2,0) {$\young(:34,1,2)$};
\node (5) at (2.5,0) {$\overset{\bij}{\mapsto}$};
\node (6) at (3,0) {$\young(:14,2,3)$};
\node (7) at (3.5,-0.15) {,};
\draw [->, >=stealth] (3,.35) -- (3,.5) -- node[above]{\scriptsize $\bij$}  (0,.5) -- (0,.35);  
\end{tikzpicture}
\end{center}
with cyclic descent sets $\{2,3\}$,\ $\{3,4\}$,\ $\{1,4\}$ and $\{1,2\}$, respectively.
\end{example}
\medskip

Of special interest are {\em horizontal strips}, which are the skew shapes all of whose connected components are single rows. 

\begin{corollary}\label{def-hor}
For every horizontal strip $\lambda/\mu$ with $n$ cells, a cyclic descent map is given by letting
\[
\cDes(T) := \{i\in[n]:\ i+1 \!\pmod n \text{ is in a lower row than $i$ in $T$}\}
\]
for $T\in\SYT(\lambda/\mu)$.
\end{corollary}

\subsection{Cyclic super-descents for words and cyclic descents for strips}\label{sec:words}

The cyclic descent map on SYT of a strip shape can be equivalently described in terms of a {\em cyclic super-descent} map on words with bi-colored letters, as follows.

Let $[m]^n$ be the set of all words $w = w_1 \cdots w_n$ of length $n$ over the alphabet $[m]$. Let $p:[m]^n\rightarrow [m]^n$ be the cyclic rotation defined by
\[
p(w_1 \cdots w_n) = w_n w_1 \cdots w_{n-1}.
\]
We now assume that the letters in the alphabet $[m]$ are colored with two colors; equivalently, partitioned into two complementary subsets. 
Let $N \subseteq [m]$ be one of these subsets.

\begin{definition}\label{def:cDes_N}
Fix a subset $N \subseteq [m]$.
\begin{enumerate}
\item
Define the {\em $N$-super-descent} map, $\Des_N:[m]^n \rightarrow 2^{[n-1]}$, by  
 \[
\Des_N(w) := \{1 \le i \le n-1 \,:\, \text{either}\ w_i > w_{i+1} \text{ or } w_i = w_{i+1} \in N\}.
\]

\item
Define the {\em cyclic $N$-super-descent} map,  $\cDes_N:[m]^n \rightarrow 2^{[n]}$, by  
\[
 \cDes_N(w) := \{1 \le i \le n \,:\, \text{either } w_i > w_{i+1} \text{ or } w_i = w_{i+1} \in N\},
\]
where index addition is modulo $n$.
\end{enumerate}
\end{definition}

\begin{example}
For $m=3$ and $N=\{2\}$, 
$\Des_N(233122)=\{3,5\}$ and $\cDes_N(233122)=\{3,5,6\}$.
\end{example}

\begin{remark}
This notion of $N$-super-descents implicitly appears 
in~\cite{BR} for $N := [k]$ and in~\cite{RS} for any $N \subseteq [n]$;
see also~\cite{Remmel}.
The super-RSK algorithm maps words $w \in [m]^n$ to pairs $(P,Q)$ of tableaux of the same shape, where $P$ is $(m-|N|,|N|)$-semistandard and $Q$ is standard.
This agrees with Definition~\ref{def:cDes_N}.1, in the sense that $\Des_N(w) = \Des(Q)$ for every $w \in [m]^n$.
\end{remark}

\begin{remark}\
\begin{enumerate}
\item
The cyclic rotation  $p$ shifts cyclic $N$-super-descent sets; namely, for every $w \in [m]^n$,
\[
\cDes_N(p(w))=1+\cDes_N(w),
\]
where addition is modulo $n$.

\item
For every subset $N \subseteq [m]$, the pair $(\cDes_N,p)$, when restricted to the set of non-constant words in $[m]^n$, is a cyclic extension (as in Definition~\ref{def:cDes}) of $\Des_N$. 
Note that constant words $w = r \cdots r$ (for $r \in [m]$) always violate the non-Escher axiom: $\cDes(w)$ is $[n]$ for $r \in N$, and is $\varnothing$ otherwise.
\end{enumerate}
\end{remark}

\medskip

Now fix a strip shape $\gamma$ with $n$ cells and $m \ge 2$ connected components.
Let $N = N_\gamma \subseteq [m]$ be the set of all $r \in [m]$ such that the $r$-th component of $\gamma$ (from the bottom up) is vertical.
Define a map $f : \SYT(\gamma) \to [m]^n$ as follows. 
For each $T \in \SYT(\gamma)$, let $f(T)$ be the word $w = w_1 \cdots w_n \in [m]^n$ where $w_i = r$ if the letter $i$ is in the $r$-th connected component of $T$, for $1 \le i \le n$.

It is easy to see that the following properties hold for every strip shape $\gamma$:
\begin{enumerate}
\item 
The map $f$ is injective.

\item 
The map $f$ is $\Des$- and $\cDes$-preserving;
that is, for all $T \in \SYT(\gamma)$
\[
\Des_{N_\gamma}(f(T))=\Des(T)
\]
and 
\[
\cDes_{N_\gamma}(f(T))=\cDes(T).
\]

\item
The map $f$ commutes with cyclic rotation; namely, for every $T\in \SYT(\gamma)$
\[
p(f(T))=f(\jdt(1+T)).
\]
\end{enumerate}

\begin{example} 
Let $\gamma=(7,4,4,4,3)/(4,3,3) = (3) \oplus (1^3) \oplus (3)$ be a strip with $m=3$ and $N_\gamma=\{2\}$. Let
\[
T=\young(::::347,:::1,:::8,:::9,256) \in \SYT(\gamma), \  \qquad f(T)=213311322. 
\]
Then $\cDes(T)=\{1,4,7,8,9\}=\cDes_{\{2\}}(213311322)$ . 
\end{example}

\medskip

For a horizontal strip with two connected components ($m = 2$ and $N = \varnothing$), we can apply the above bijection to obtain a formula for the number of tableaux with a given cyclic descent set.

\begin{proposition}\label{prop:cDes-2-rows}
Let $(k) \oplus (n-k)$ be a horizontal strip with two components, containing $k$ cells in the lower component and $n-k$ cells in the upper one, with $1 \le k \le n-1$. 
For each nonempty subset $J = \{j_1 < \ldots < j_t\} \subseteq [n]$,
define the sequence of cyclic differences $(d_1, \ldots, d_t)$ by
\[
\begin{aligned}
d_1 &:= j_1 - j_t + n, \\
d_i &:= j_i - j_{i-1} \qquad (2 \le i \le t).
\end{aligned}
\]
Then 
\[
\sum_{k=1}^{n-1} \card{\{T \in \SYT((k) \oplus (n-k)) \,:\, \cDes(T) = J\}} \, q^{k}
= \prod_{i=1}^{t} 
\sum_{j=1}^{d_i-1} q^j
\]
In particular, a subset $J \subseteq [n]$ as above is a cyclic descent set for some $T \in \SYT((k) \oplus (n-k))$ if and only if 
$1 \le \card{J} \le \min(k,n-k)$ and
the elements of $J$ are cyclically non-adjacent (namely, $d_i \ge 2$ for all $i$).
\end{proposition}

\begin{proof}
Apply the map $f$ defined above to encode each $T \in \SYT((k) \oplus (n-k))$ by a word 
$w = w_1 \cdots w_n \in [2]^n$ where, for each $1 \le i \le n$,
\[
w_i := 
\begin{cases}
1, & \text{if entry } i \text{ is in the lower row of } T; \\
2, & \text{if entry } i \text{ is in the upper row of } T.
\end{cases}
\]
This encoding is a bijection from $\SYT((k) \oplus (n-k))$ to the set of words $w \in [2]^n$ in which $1$ appears $k$ times (and $2$ appears $n-k$ times), and this bijection preserves cyclic descents. It thus suffices to study $\cDes$ for such words.
Clearly, $\cDes(w) = J = \{j_1 < \ldots < j_t\}$ if and only if each cyclic interval $w_{j_{i-1}+1} \cdots w_{j_i}$ has the form $1^{c_i} 2^{d_i - c_i}$ for some $1 \le c_i \le d_i-1$, with $c_1 + \ldots + c_t = k$.
This yields the claimed formula.

Finally, to determine when a subset $J \subseteq [n]$ appears as a cyclic descent set, note that the factors in the RHS of the formula have the form $q + \ldots + q^{d_i-1}$.
Thus the coefficient of $q^k$ in the LHS of the formula is nonzero if and only if $d_i \ge 2$ for each $1 \le i \le t$ and also $t \le k \le n-t$, which is equivalent to $t \le \min(k,n-k)$.
\end{proof}

Another enumerative application of this bijection
will be presented in Section~\ref{subsec:2rows_gf}.

\section{Two-row straight shapes}
\label{sec:2rows}

In this section we provide an explicit combinatorial description of a cyclic descent map
for two-row straight shapes, proving Theorem~\ref{thm:2rows_intro}.

\subsection{From two-row shapes to lattice paths}\label{sec:paths}
It will be useful to view two-row SYT as lattice paths on the plane.
Let $\AAA_n$ denote the set of all lattice paths with $n$ steps
$U=(1,1)$ and $D=(1,-1)$ starting at the origin. Let $\AAA_n^h$
denote the set of paths in $\AAA_n$ whose ending height (that is,
the $y$-coordinate of its endpoint) is $h$. Define the {\em peak set} of a path $P\in\AAA_n$ as
\[
\Peak(P)=\{i: \text{the $i$th step of $P$ is a $U$ followed by a $D$}\}\subseteq[n-1].
\]
Let $\PPP_n\subset\AAA_n$ be the set of paths that do not go
below the $x$-axis, and let $\PPP_n^h=\PPP_n\cap\AAA_n^h$ be those that end at
height $h$. Let $\QQQ_n=\AAA_n\setminus\PPP_n$, and define $\QQQ_n^h$
similarly.

Consider the map $\Gamma$ from two-row SYT (possibly of skew shape) to lattice paths obtained by letting the $i$th step of the path be a $U$ if $i$ is the first row of the tableau, and a $D$ otherwise. Clearly, $\Gamma$ sends the statistic $\Des$ on tableaux to the statistic $\Peak$ on paths.
The map $\Gamma$ restricts to a bijection between $\SYT(n-k,k)$ and $\PPP_n^{n-2k}$. 
As a consequence, finding a cyclic descent map for $(n-k,k)$ as in Definition~\ref{def:extension} is equivalent to constructing a {\em cyclic peak map} on $\PPP_n^{n-2k}$, defined as follows.

\begin{definition}\label{def:cPeak} Let $\mathcal{B}\subseteq\AAA_n$. A {\em cyclic peak map} on $\mathcal{B}$ is a map $\cPeak:\PPP_n^{n-2k}\to 2^{[n]}$ such that
\begin{enumerate}[(i)]
\item $\cPeak(P)\cap[n-1]=\Peak(P)$ for all $P\in\mathcal{B}$ (i.e., $\cPeak$ {\em extends} $\Peak$);
\item the multiset $\multiset{\cPeak(P):P\in\mathcal{B}}$ is closed under cyclic rotation modulo $n$.
\end{enumerate} 
\end{definition}

Indeed, once a cyclic peak map $\cPeakP$ on $\PPP_n^{n-2k}$ has been constructed, the map $\cDes:\SYT(n-k,k)\to2^{[n]}$ defined by
$\cDes(T):=\cPeakP(\Gamma(T))$ for $T\in\SYT(n-k,k)$ will be a cyclic descent map for $(n-k,k)$.

We start with the much simpler problem of defining cyclic peak map $\cPeakA$ on $\AAA_n^h$, for any fixed $h$. Unlike $\PPP_n^h$, the set $\AAA_n^h$ is closed under cyclic shifts, namely the operation of moving the last step to the beginning of the path. Thus, for $P\in\AAA_n^h$, we can simply define
$$\cPeakA(P)=\begin{cases} \Peak(P)\sqcup\{n\} & \text{if $P$ ends with a $U$ and starts with a $D$,} \\ \Peak(P) & \text{otherwise.}\end{cases}$$
Clearly, $\cPeakA$ on $\AAA_n^h$ satisfies Definition~\ref{def:cPeak}. 

Next we relate the sets $\AAA_n$ and $\PPP_n$ by describing an injection $\varphi:\AAA_n^{n-2k+2}\to\AAA_n^{n-2k}$ with the property that 
the set of paths that are not in its image is $\AAA_n^{n-2k}\setminus \QQQ_n^{n-2k}=\PPP_n^{n-2k}$.

\begin{lemma}\label{lem:varphi} For $1\le k\le n/2$, there is a $\Peak$-preserving bijection
$$\varphi:\AAA_n^{n-2k+2}\longrightarrow \QQQ_n^{n-2k}\subset \AAA_n^{n-2k}.$$
\end{lemma}

\begin{proof} Given $P\in\AAA_n^{n-2k+2}$, match $U$ and
$D$ steps that face each other, meaning that they are at the same height, $U$ is to the left of $D$, and the horizontal line
segment connecting their midpoints stays below the path. (Viewing $U$ and $D$ steps as opening and closing parentheses, respectively, 
this is equivalent to matching parentheses properly.) Note that the sequence of unmatched steps, from left to right, is of the form $D^iU^j$, and that $j>0$ because $n-2k+2>0$. Let $\varphi(P)$ be the path obtained from $P$ by changing its leftmost unmatched $U$ into a $D$. See Figure~\ref{fig:varphi} for an example. It is clear that $\Peak(P)=\Peak(\varphi(P))$, since the changed step is not part of a peak in either $P$ or $\varphi(P)$.

Since the leftmost unmatched $U$ of $P$ must start at height $0$ or lower, we have that
$\varphi(P)\in\QQQ_n^{n-2k}$. To show that $\varphi$ is a bijection between $\AAA_n^{n-2k+2}$ and $\QQQ_n^{n-2k}$,
we describe its inverse. Given a path in $\QQQ_n^{n-2k}$, match again $U$ and
$D$ steps that face each other, and then change the rightmost unmatched $D$,
which must exist because the path goes below the $x$-axis, into a $U$.
\end{proof}

\begin{figure}[htb]
\centering
    \begin{tikzpicture}[thick, scale=0.35]
    \def\U{(1,1)}     \def\D{(1,-1)}
    \draw (0,0) coordinate(d0)
      -- ++\U coordinate(d1)
      -- ++\D coordinate(d2)
      -- ++\D coordinate(d3)
      -- ++\U coordinate(d4)
      -- ++\U coordinate(d5) 
      -- ++\D coordinate(d6)
      -- ++\D  coordinate(d7)
      -- ++\U coordinate(d8)
      -- ++\U coordinate(d9)
      -- ++\U  coordinate(d10)
      -- ++\U  coordinate(d11)
      -- ++\D coordinate(d12) 
      -- ++\U coordinate(d13);
      \draw[densely dotted] (0,0)--(13,0);
      \draw[dotted,orange] (.5,0.5)--(1.5,0.5);
      \draw[dotted,orange] (4.5,0.5)--(5.5,0.5);
      \draw[dotted,orange] (3.5,-0.5)--(6.5,-0.5);
      \draw[dotted,orange] (10.5,2.5)--(11.5,2.5);
      \draw[green,line width=2.0pt] (d7)--(d8);
      \foreach \x in {0,...,13} {
        \draw (d\x) circle (1.5pt);
      }
      \draw (15,1) node [label=$\varphi$] {$\mapsto$};
      \begin{scope}[shift={(17,0)}]
    \draw (0,0) coordinate(d0)
      -- ++\U coordinate(d1)
      -- ++\D coordinate(d2)
      -- ++\D coordinate(d3)
      -- ++\U coordinate(d4)
      -- ++\U coordinate(d5) 
      -- ++\D coordinate(d6)
      -- ++\D  coordinate(d7)
      -- ++\D coordinate(d8)
      -- ++\U coordinate(d9)
      -- ++\U  coordinate(d10)
      -- ++\U  coordinate(d11)
      -- ++\D coordinate(d12) 
      -- ++\U coordinate(d13);
      \draw[densely dotted] (0,0)--(13,0);
      \draw[dotted,orange] (.5,0.5)--(1.5,0.5);
      \draw[dotted,orange] (4.5,0.5)--(5.5,0.5);
      \draw[dotted,orange] (3.5,-0.5)--(6.5,-0.5);
      \draw[dotted,orange] (10.5,0.5)--(11.5,0.5);
      \draw[green,line width=2.0pt] (d7)--(d8);
      \foreach \x in {0,...,13} {
        \draw (d\x) circle (1.5pt);
      }
      \end{scope}
    \end{tikzpicture}
\caption{An example of the bijection $\varphi:\AAA_n^{n-2k+2}\longrightarrow \QQQ_n^{n-2k}$.}
\label{fig:varphi}
\end{figure}

Unfortunately, $\varphi$ does not preserve $\cPeakA$ in general: there are paths $P\in\AAA_n^{n-2k+2}$ with $n\in\cPeakA(\varphi(P))$ but $n\notin\cPeakA(P)$.
We will require our map $\cPeakP$ on $\PPP_n^{n-2k}$ to satisfy that, for all $J\subseteq[n]$,
\begin{equation}\label{eq:cPeak'}
|\{P\in\PPP_n^{n-2k}:\cPeakP(P)=J\}|=|\{P\in\AAA_n^{n-2k}:\cPeakA(P)=J\}|-|\{P\in\AAA_n^{n-2k+2}:\cPeakA(P)=J\}|.
\end{equation}
This requirement, together with the fact that $\multiset{\cPeakA(P):P\in\AAA_n^h}$ is closed under cyclic rotation for every $h$,
will imply that condition (ii) of Definition~\ref{def:cPeak} holds for $\cPeakP$. 

\subsection{Cyclic descents on two-row straight shapes}\label{sec:combinatorial2row}

In the previous section we have reduced the construction of a cyclic descent map for $(n-k,k)$ to the construction of a cyclic peak set $\cPeakP$ on $\PPP_n^{n-2k}$ that extends $\Peak$ and satisfies Equation~\eqref{eq:cPeak'}.
In the rest of this section we assume that $2\le k\le n/2$, since we know by Proposition~\ref{prop:ribbons} that no cyclic descent map 
exists for $(n-1,1)$.

Fix $I\subseteq[n-1]$. Paths in $\AAA_n^{n-2k}$ with peak set $I$ can be partitioned as follows:
\begin{multline}\label{eq:union}
\{P\in\AAA_n^{n-2k}:\Peak(P)=I\}\ =\{P\in\PPP_n^{n-2k}:\Peak(P)=I\}\sqcup\{P\in\QQQ_n^{n-2k}:\Peak(P)=I\}\\
=\{P\in\PPP_n^{n-2k}:\cPeakA(P)=I\}\sqcup\{P\in\QQQ_n^{n-2k}:\cPeakA(P)=I\}\sqcup\{P\in\QQQ_n^{n-2k}:\cPeakA(P)=I\sqcup\{n\}\},
\end{multline}
where the last equality uses the fact that every $P\in\PPP_n^{n-2k}$
starts with a $U$ step, and thus satisfies $n\notin\cPeakA(P)$.

\begin{figure}[htb]
\centering
\begin{tikzpicture}[scale=1.1]
\draw (0,0) rectangle (6,2.5);
\fill[blue!30!white,postaction={pattern=north east lines}] (0,1) rectangle (6,1.5);
\fill[blue!30!white] (0,1.5) rectangle (6,2.5);
\draw (0,1) rectangle (6,2.5);
\filldraw[fill=yellow!40!white] (0,0) rectangle (6,1);
\draw (3,1.75) node {$\{P\in\AAA_n^{n-2k+2}:\cPeakA(P)=I\}$};
\draw (3,.5) node {$\{P\in\AAA_n^{n-2k+2}:\cPeakA(P)=I\sqcup\{n\}\}$};
\draw [decorate,decoration={brace,mirror,amplitude=5pt},xshift=4pt] (6,0) -- (6,2.5);
\draw[->] (6.5,1.25)-- node[above]{$\varphi$} node[below]{$\sim$} (7.5,1.25);
\draw (3,3.25) node {$\{P\in\AAA_n^{n-2k+2}:\Peak(P)=I\}$};

\begin{scope}[shift={(8,0)}]
\draw [decorate,decoration={brace,amplitude=5pt},xshift=-4pt] (0,0) -- (0,2.5);
\draw (0,0) rectangle (6,3.5);
\filldraw[fill=red!30!white] (0,2.5) rectangle (6,3.5);
\filldraw[fill=blue!30!white] (0,1.5) rectangle (6,2.5);
\fill[yellow!40!white,postaction={pattern=north east lines}] (0,1) rectangle (6,1.5);
\fill[yellow!40!white] (0,0) rectangle (6,1);
\draw (0,0) rectangle (6,1.5);
\draw (3,3) node {$\{P\in\PPP_n^{n-2k}:\cPeakA(P)=I\}$};
\draw (3,2) node {$\{P\in\QQQ_n^{n-2k}:\cPeakA(P)=I\}$};
\draw (3,.75) node {$\{P\in\QQQ_n^{n-2k}:\cPeakA(P)=I\sqcup\{n\}\}$};
\draw (3,4.25) node {$\{P\in\AAA_n^{n-2k}:\Peak(P)=I\}$};
\end{scope}
\end{tikzpicture}
\caption{The behavior of $\varphi$ with respect to $\cPeakA$.}
\label{fig:varphi_blocks}
\end{figure}

By Lemma~\ref{lem:varphi}, $\varphi$ gives a bijection between
$\{P\in\AAA_n^{n-2k+2}:\Peak(P)=I\}$ and
$\{P\in\QQQ_n^{n-2k}:\Peak(P)=I\}$.
Even though $\varphi$ does not preserve $\cPeakA$ in general
, the next lemma describes precisely when it does. The paths with $\cPeakA(\varphi(P))=\cPeakA(P)\sqcup\{n\}$ correspond to the striped region in Figure~\ref{fig:varphi_blocks}.

\begin{lemma}\label{lem:cPeak}
For $2\le k\le n/2$ and $P\in \AAA_n^{n-2k+2}$, 
$$\cPeakA(\varphi(P))=\begin{cases} \cPeakA(P)\sqcup\{n\} & \text{if $P$ is of the form $P=URU$ for some $R\in\PPP_{n-2}^{n-2k}$},\\
\cPeakA(P) & \text{otherwise.} \end{cases}$$
\end{lemma}

\begin{proof}
Recall from Lemma~\ref{lem:varphi} that $\varphi$ is $\Peak$-preserving, and $\varphi(P)$ is obtained from $P\in \AAA_n^{n-2k+2}$ by
changing the leftmost unmatched $U$ of $P$ into a $D$. Since
$n-2k\ge 0$, the leftmost unmatched $U$ is never the last step
of $P$. Therefore, $\varphi$ does not change the last step.

It follows that in order to have
$\cPeakA(\varphi(P))\neq\cPeakA(P)$, $\varphi$ must change the first
step of $P$ (from a $U$ to a $D$) and $P$ must end with a $U$. For
the first step of $P$ to be the leftmost unmatched $U$, we must
have $P=URU$ for some $R\in\PPP_{n-2}^{n-2k}$. In this case,
$\varphi(P)=DRU$, and thus $n\in\cPeakA(\varphi(P))$ but
$n\notin\cPeakA(P)$.
\end{proof}

Our strategy to define $\cPeakP$ satisfying Equation~\eqref{eq:cPeak'} will be to declare that $n\in\cPeakP(P)$ for certain paths $P\in\PPP_n^{n-2k}$, making the number of paths with this property equal to the number of paths $P\in \AAA_n^{n-2k+2}$ for which $\cPeakA(P)$ and $\cPeakA(\varphi(P))$ disagree.

Provided that $\cPeakP$ extends $\Peak$, in order for Equation~\eqref{eq:cPeak'} to hold for all $J\subseteq[n]$, it is enough if it holds for sets $J$ containing $n$. This is because, for $I\subseteq[n-1]$, having $\cPeakP(P)$ equal to $I$ or $I\sqcup\{n\}$  is equivalent to having $\Peak(P)=I$, and
$|\{P\in\PPP_n^{n-2k}:\Peak(P)=I\}|=|\{P\in\AAA_n^{n-2k}:\Peak(P)=I\}|-|\{P\in\AAA_n^{n-2k+2}:\Peak(P)=I\}|$ by Lemma~\ref{lem:varphi}.

For $J=I\sqcup\{n\}$, Equation~\eqref{eq:cPeak'} becomes
\begin{multline}\label{eq:size}
|\{P\in\PPP_n^{n-2k}:\cPeakP(P)=I\sqcup\{n\}\}|\\
=|\{P\in\AAA_n^{n-2k}:\cPeakA(P)=I\sqcup\{n\}\}|-|\{P\in\AAA_n^{n-2k+2}:\cPeakA(P)=I\sqcup\{n\}\}|\\
=|\{P\in\QQQ_n^{n-2k}:\cPeakA(P)=I\sqcup\{n\}\}|-|\{P\in\AAA_n^{n-2k+2}:\cPeakA(P)=I\sqcup\{n\}\}|\\
=|\{P: P=URU \text{ for some } R\in\PPP_{n-2}^{n-2k} \text{ and }
\Peak(P)=I\}|,
\end{multline}
where we have used the decomposition~\eqref{eq:union} and Lemma~\ref{lem:cPeak}.

To define $\cPeakP$ satisfying condition \eqref{eq:size}, we will find a subset of $\{P\in\PPP_n^{n-2k}:\Peak(P)=I\}$ of size
equal to the right-hand side of \eqref{eq:size} (corresponding to the striped region in Figure~\ref{fig:varphi_blocks}); and then let $\cPeakP(P):=I\cup\{n\}$ for $P$ in the subset, and $\cPeakP(P):=I$ otherwise. To find such subset, we construct
an injection
$$\psi:\{P: P=URU \text{ for some } R\in\PPP_{n-2}^{n-2k} \text{ and } \Peak(P)=I\}\longrightarrow \{P\in\PPP_n^{n-2k}:\Peak(P)=I\}$$
and let the aforementioned subset be its image.

Given $P$ in the domain of the map, let $\psi(P)$ be the path obtained by changing the
last occurrence of $DU$ in $P$ into $DD$. Note that $P$ must
contain a $D$ because $n-2k+2<n$, since $k>1$, and thus it must contains some $DU$ since the last letter of paths in the domain is $U$. Also, we have
$\psi(P)\in\PPP_n^{n-2k}$, and $\psi$ preserves $\Peak$, since the
last $DU$ of $P$ cannot be followed by a $D$. Figure~\ref{fig:psi} shows an example of this construction. 
To see that $\psi$
is an injection, note that $P$ can be recovered from $\psi(P)$ by
changing the rightmost $D$ into a $U$. 

\begin{figure}[htb]
\centering
    \begin{tikzpicture}[thick, scale=0.35]
    \def\U{(1,1)}     \def\D{(1,-1)}
    \draw (0,0) coordinate(d0)
      -- ++\U coordinate(d1)
      -- ++\U coordinate(d2)
      -- ++\U coordinate(d3)
      -- ++\D coordinate(d4)
      -- ++\U coordinate(d5) 
      -- ++\D coordinate(d6)
      -- ++\D  coordinate(d7)
      -- ++\U coordinate(d8)
      -- ++\U coordinate(d9);
     \draw[densely dotted] (0,0)--(9,0);
      \draw[blue,line width=2.0pt] (d6)--(d7);
      \draw[blue,line width=2.0pt] (d7)--(d8);
      \foreach \x in {0,...,9} {
        \draw (d\x) circle (1.5pt);
      }
      \draw (11,1) node [label=$\psi$] {$\mapsto$};
      \begin{scope}[shift={(13,0)}]
       \draw (0,0) coordinate(d0)
      -- ++\U coordinate(d1)
      -- ++\U coordinate(d2)
      -- ++\U coordinate(d3)
      -- ++\D coordinate(d4)
      -- ++\U coordinate(d5) 
      -- ++\D coordinate(d6)
      -- ++\D  coordinate(d7)
      -- ++\D coordinate(d8)
      -- ++\U coordinate(d9);
      \draw[densely dotted] (0,0)--(9,0);
      \draw[blue,line width=2.0pt] (d6)--(d7);
      \draw[blue,line width=2.0pt] (d7)--(d8);
      \foreach \x in {0,...,9} {
        \draw (d\x) circle (1.5pt);
      }
      \end{scope}
    \end{tikzpicture}
\caption{An example of the injection $\psi$.}
\label{fig:psi}
\end{figure}
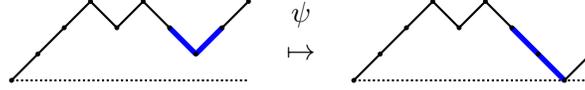

For $P\in\PPP_n^{n-2k}$ with $\Peak(P)=I$, we declare that $n\in
\cPeakP(P)$ if and only if $P$ is in the image of $\psi$, which
can be described as the set of paths in $\PPP_n^{n-2k}$ where the
rightmost $D$ is preceded by another $D$, and after changing that
rightmost $D$ into a $U$, the resulting path is of the form
$P=URU$ with $R\in\PPP_{n-2}$. This is equivalent to the following definition. 
Recall that a {\em return} of a path in $\PPP_n$ is a $D$ step that ends on the $x$-axis.

\begin{definition}\label{def:cPeakP}
For $P\in\PPP_n^{n-2k}$ with $2\le k\le n/2$, let $n\in\cPeakP(P)$ if and only if the
rightmost $D$ of $P$ is preceded by another $D$, and $P$ has no returns before the rightmost $D$.
\end{definition}

We point out that, from the above definition, it is clear that having $n\in\cPeakP(P)$ forces $1,n-1\notin\cPeakP(P)$. This is consistent with the fact that a path cannot have two peaks in adjacent positions.
Via the bijection $\Gamma$ between $\SYT(n-k,k)$ and $\PPP_n^{n-2k}$, the statistic $\cPeakP$ translates into the following statistic $\cDes$ on $\SYT(n-k,k)$.

\begin{definition}\label{def:cDes2rowleft}
For $T\in\SYT(n-k,k)$ with $2\le k\le n/2$, let $n\in\cDes(T)$ if and only if both of the following conditions
hold:
\begin{enumerate}
\item the last two entries in the second row of $T$ are
consecutive, that is, $T_{2,k}=T_{2,k-1} + 1$, \item for every
$1<i<k$,  $T_{2,i-1} > T_{1,i}$.
\end{enumerate}
\end{definition}

\begin{example} According to Definition~\ref{def:cDes2rowleft}, we have $9\in\cDes\left(\young(12359,4678)\right)$, but $9\notin\cDes\left(\young(13469,2578)\right)$.
\end{example}

\begin{proof}[Proof of Theorem~\ref{thm:2rows_intro}]
Theorem~\ref{thm:2rows_intro} states that, for $2\le k\le n/2$, Definition~\ref{def:cDes2rowleft} gives a cyclic descent map $\cDes$ for $(n-k,k)$. By construction, this definition is equivalent to setting $\cDes(T)=\cPeakP(\Gamma(T))$ for all $T\in\SYT(n-k,k)$.
Having shown that $\cPeakP$ on $\PPP_n^{n-2k}$ satisfies Definition~\ref{def:cPeak}, it follows that $\cDes$ on $\SYT(n-k,k)$ satisfies Definition~\ref{def:extension}.
\end{proof}

Note that condition 2 in Definition~\ref{def:cDes2rowleft} is equivalent to the
fact that, after removing $T_{2,k}$ and $T_{1,1}=1$ from $T$ and applying jdt, the resulting tableau has shape $(n-k-1,k-1)$.

We conclude this section by comparing Definition~\ref{def:cDes2rowleft} with previous definitions of cyclic descent sets.

\begin{remark}\begin{enumerate}
\item For the shape $(n-2,2)$, Definition~\ref{def:cDes2rowleft} coincides with
the only cyclic descent map for hooks plus an internal cell described in
Theorem~\ref{thm:hooks_plus1_intro}.

\item On two-row rectangular shapes, Definition~\ref{def:cDes2rowleft}
coincides with Rhoades' definition~\cite{Rhoades}, restated right after Theorem~\ref{Rhoades_thm}. 
Indeed, for even $n$ and $T\in\SYT(n/2,n/2)$, Rhoades' definition declares that $n\in\cDes(T)$ if and only if, when applying $\jdt$ to $-1+T$, the path that $n$ follows goes east at every step until the last step, when it goes south. This happens if and only if $T_{2,i-1}>T_{1,i}$ for $1<i\le n/2$, which is equivalent to condition 2 of Definition~\ref{def:cDes2rowleft} plus the fact that $T_{2,n/2-1}=n-1$, which translates into condition~1. 
\end{enumerate}
\end{remark}

\subsection{The bijection $\bij$ for two-row straight shapes}
\label{sec:2rows_action}

Let $1<k\le n/2$. We will define a bijection $\bij:\SYT(n-k,k)\to\SYT(n-k,k)$ with the property that, with $\cDes$ as given by Definition~\ref{def:cDes2rowleft},
\begin{equation}\label{eq:rot2rows} \cDes(\bij T)=1+\cDes(T) \end{equation}
for all $T\in \SYT(n-k,k)$, with addition modulo $n$.

We first construct a bijection $\rho:\PPP^{n-2k}_n\to\PPP^{n-2k}_n$ such that
$\cPeakP(\rho P)=1+\cPeakP(P)$ for all $P\in\PPP^{n-2k}_n$, with addition modulo $n$. Let $\PPP=\bigcup_{n\ge0}\PPP_n$ and $\PPP^0=\bigcup_{n\ge0}\PPP^0_n$.
Given $P\in\PPP^{n-2k}_n$, define $\rho P$ by considering four cases:
\begin{enumerate}
	\item If $n\in\cPeakP(P)$, let $\rho P$ be the path obtained by moving the last $D$ of $P$ and inserting it as the second step of the path. In other words, write $P=UQDDU^i$ for some $i\ge0$ and some path $Q$, and let $\rho P=UDQDU^i$. 
	\item If $n-1\in\cPeakP(P)$, write $P=QDU^{i+1}D$ for some $i\ge0$ and some path $Q\in\PPP$, and let $\rho P=UQDDU^i$.
	\item If $n-1,n\notin\cPeakP(P)$ and the last $D$ of $P$ is preceded by a $U$, let $\rho P$ be the path obtained by simply moving the last step of $P$ (which is necessarily a $U$) to the beginning.
	\item If $n-1,n\notin\cPeakP(P)$ and the last $D$ is preceded by a $D$, write $P=RUQDDU^i$ for some $i\ge0$, $R\in\PPP^0$ and $Q\in\PPP$ (that is, the last return of $P$ not including the final $D$ occurs at the end of $R$), and let $\rho P=URDQDU^i$. (Note that $R$ is never empty because, since $n\notin\cPeakP(P)$, $P$ must have some return before the last $D$.)
\end{enumerate}

To prove that $\rho$ is a bijection, we describe its inverse in each of the four cases above, for given $P'\in\PPP^{n-2k}_n$, as follows:
\begin{enumerate}
	\item If $1\in\cPeakP(P')$, then $\rho^{-1}P'$ is the path obtained by moving the first $D$ of $P'$ (which is the second step) and inserting it right after the last $D$ of $P'$.
	\item If $n\in\cPeakP(P')$, write $P'=UQDDU^i$ for some $i\ge0$ and some path $Q\in\PPP$, and let $\rho^{-1}P'=QDU^{i+1}D$.
	\item If $1,n\notin\cPeakP(P')$ and $P'$ has no returns, then $\rho^{-1}P'$ is the path obtained by simply moving the first step of $P$ (which is necessarily a $U$) to the end.
	\item If $1,n\notin\cPeakP(P)$ and $P'$ has returns, write $P'=URDQDU^i$ for some $i\ge0$, $R\in\PPP^0$ and $Q\in\PPP$ (that is, the first return of $P'$ is immediately before $Q$), and let $\rho^{-1}P'=RUQDDU^i$.
\end{enumerate}

It is easy to check from the above definition that $\cPeakP(\rho P)=1+\cPeakP(P)$ in all four cases. Letting $\bij=\Gamma^{-1}\circ\rho\circ\Gamma$, we obtain a bijection $\bij:\SYT(n-k,k)\to\SYT(n-k,k)$ satisfying Equation~\eqref{eq:rot2rows}.

Next we describe $\bij$ directly in terms of SYT as follows. The three cases below correspond to the four cases above, with 1 and 4 consolidated into one new case 1'. For $T\in\SYT(n-k,k)$, let $x$ denote the last entry in the second row of $T$, and let $1+T_{\le x}$ be the tableau obtained by adding $1$ modulo $x$ to the entries of $T$ that are smaller or equal to $x$ (and leaving the entries bigger than $x$ unchanged). Given $T\in\SYT(n-k,k)$, construct $\bij T$ by considering the following cases, which depend on whether the last two entries of the second row of $T$ are consecutive or not:
\begin{enumerate}
	\item[1'.] If they are consecutive, let $\bij T=\jdt(1+T_{\le x})$.
	\item[2.] If they are not consecutive and $n$ is in the first row of $T$, let $\bij T=\jdt(1+T)$.
	\item[3.]  If they are not consecutive and $n$ is in the second row of $T$, let $\bij T$ be the tableau obtained from $1+T$ as follows: switch $1$ (which is the last entry in the second row of $1+T$) with $y+1$, where $y$ is the entry immediately to the left of $1$; then apply $\jdt$ as usual.
\end{enumerate}

\begin{remark}
	When $n$ is even and $k=n/2$, the above definition of $\bij T$ coincides with $\jdt(1+T)$ in all cases.
\end{remark}

\begin{example} Below are two orbits of the $\bbz$-actions on $\SYT(5,4)$ and $\SYT(6,3)$ generated by $\bij$, respectively. Cyclic descents in each SYT are marked in boldface \textcolor{red}{\bf red}.
	\begin{multline*}
T=\young(\rone\rthree56\rseven,2489)\overset{\bij}{\mapsto}\young(1\rtwo\rfour7\reight,3569)\overset{\bij}{\mapsto}\young(12\rthree\rfive\rnine,4678)\overset{\bij}{\mapsto}\young(\rone3\rfour\rsix9,2578)\overset{\bij}{\mapsto}\young(1\rtwo\rfive\rseven9,3468)\\
	\overset{\bij}{\mapsto}\young(12\rthree\rsix\reight,4579)\overset{\bij}{\mapsto}\young(123\rfour\rseven,568\rnine)\overset{\bij}{\mapsto}\young(\rone34\rfive\reight,2679)\overset{\bij}{\mapsto}\young(1\rtwo45\rsix,378\rnine){\mapsto}\young(\rone\rthree56\rseven,2489)=T
	\end{multline*}
	\begin{multline*}
	T=\young(\rone3\rfour\rsix89,257)\overset{\bij}{\mapsto}\young(1\rtwo4\rfive\rseven9,368)\overset{\bij}{\mapsto}\young(12\rthree5\rsix\reight,479)\overset{\bij}{\mapsto}\young(123\rfour6\rseven,58\rnine)\overset{\bij}{\mapsto}\young(\rone34\rfive7\reight,269)\\
	\overset{\bij}{\mapsto}\young(1\rtwo45\rsix\rnine,378)\overset{\bij}{\mapsto}\young(\rone\rthree56\rseven9,248)\overset{\bij}{\mapsto}\young(1\rtwo\rfour67\reight,359)\overset{\bij}{\mapsto}\young(12\rthree\rfive8\rnine,467)\overset{\bij}{\mapsto}\young(\rone3\rfour\rsix89,257)=T
	\end{multline*}
\end{example}

Despite the above two examples, it is not true in general that $\bij^n T=T$ for all $T\in\SYT(n-k,k)$, and so $\bij$ does not generate a $\ZZ_n$-action. For example, letting $T=\young(13479,2568)$, we have
$\bij^9 T=\young(13467,2589)\neq T$. In this case, $\bij^6 T=T$.

\subsection{Generating functions for $\cdes$ over two-row straight shapes}
\label{subsec:2rows_gf}

For a positive integer $n$, denote $n_1 := \lfloor n/2 \rfloor$ and $n_2 := \lceil n/2 \rceil$, so that $n_1 \le n_2$ and $n_1 + n_2 = n$.
Let $\TTT$ be the set of all straight SYT with $n$ cells whose shape has at most two rows.
Recall the definition $\des(T) := \card{\Des(T)}$ of the {\em descent number} of $T \in \TTT$.
The following result appears, for example, in~\cite[Theorem 3.2]{BBES}).
\begin{proposition}\label{t.gf_des} We have
\[
\TTT^{\des}(q) := \sum_{T \in \TTT} q^{\des(T)}
= \sum_{d=0}^{n_1} \binom{n_1}{d} \binom{n_2}{d} q^d.
\]
\end{proposition}
We now prove an analogous formula for the {\em cyclic descent number} $\cdes(T) = \card{\cDes(T)}$. To this end, we replace the pair of hook shapes in $\TTT$, $h_0 = (n)$ and $h_1 = (n-1,1)$, which do not possess cyclic descent extensions, with the single shape $h' = (1) \oplus (n-1)$, and denote by $\hTTT$ the resulting set of SYT. 
Note that $\hTTT^{\des}(q) = \TTT^{\des}(q)$, since there is an obvious $\Des$-preserving bijection between $\SYT(h')$ and $\SYT(h_0) \cup \SYT(h_1)$.

\begin{theorem}\label{t.gf_cdes} We have
\[
\hTTT^{\cdes}(q) := \sum_{T \in \hTTT} q^{\cdes(T)}
= \sum_{d=1}^{n_1} \frac{n}{d} \binom{n_1-1}{d-1} \binom{n_2-1}{d-1} q^d.
\]
\end{theorem}

We present three proofs of this theorem. The first applies a differential equation;
the second applies the bijection to words from Section~\ref{sec:words};
and the third applies the bijection to lattice paths from Section~\ref{sec:paths}.

\begin{proof}[First Proof]
Let us first restate the formula for $\hTTT^{\des}(q)$ from Proposition~\ref{t.gf_des}.
\begin{align*}
\hTTT^{\des}(q) 
&= \sum_{d=0}^{n_1} \binom{n_1}{d} \binom{n_2}{d} q^d 
= \sum_{d=0}^{n_1} \sum_{k=0}^{d} \binom{n_1}{d} \binom{n_2}{d} \binom{d}{k} (q-1)^k \\
&= \sum_{k=0}^{n_1} \sum_{d=k}^{n_1} \binom{n_1}{d} \binom{n_2}{k} \binom{n_2 - k}{n_2 - d} (q-1)^k 
= \sum_{k=0}^{n_1} \binom{n_2}{k} \binom{n - k}{n_2} (q-1)^k.
\end{align*}
By \cite[Lemma 2.4 and Remark 2.5]{ARR},
\[
\frac{n \hTTT^{\des}(q)}{(1-q)^{n+1}} 
= \frac{d}{dq} \left[ \frac{\hTTT^{\cdes}(q)}{(1-q)^{n}} \right]
\]
and, moreover, the polynomial $\hTTT^{\cdes}(q)$ has no constant term.
Thus
\begin{align*}
\hTTT^{\cdes}(q)
&= (1-q)^n \int_{0}^{q} \frac{n \hTTT^{\des}(t)}{(1-t)^{n+1}} dt \\
&= (1-q)^n \int_{0}^{q} \sum_{k=0}^{n_1} \binom{n_2}{k} \binom{n - k}{n_2} (t-1)^k \cdot n (1-t)^{-(n+1)} dt \\
&= (1-q)^n \sum_{k=0}^{n_1} \binom{n_2}{k} \binom{n - k}{n_2} \frac{(-1)^k n}{n-k} \left[ (1-q)^{k-n} - 1 \right] \\
&= \sum_{k=0}^{n_1} \binom{n_2}{k} \binom{n - k - 1}{n_2 - 1} \frac{(-1)^k n}{n_2} \left[ (1-q)^{k} - (1-q)^{n} \right] \\
&= \sum_{k=0}^{n_1} \sum_{d=\max(k,1)}^{n_1} \binom{n_2}{k} \binom{n_1 - 1}{d - 1} \binom{n_2 - k}{n_2 - d} \frac{(-1)^k n}{n_2} \left[ (1-q)^{k} - (1-q)^{n} \right] \\
&= \sum_{d=1}^{n_1} \sum_{k=0}^{d} \binom{n_1 - 1}{d - 1} \binom{n_2}{d} \binom{d}{k} \frac{(-1)^k n}{n_2} \left[ (1-q)^{k} - (1-q)^{n} \right] \\
&= \sum_{d=1}^{n_1} \binom{n_1 - 1}{d - 1} \binom{n_2}{d} \frac{n}{n_2} \left[ ((q-1)+1)^d - (1-1)^d (1-q)^n \right] \\
&= \sum_{d=1}^{n_1} \frac{n}{d} \binom{n_1 - 1}{d - 1} \binom{n_2 - 1}{d - 1} q^d.
\qedhere
\end{align*}
\end{proof}

\begin{proof}[Second Proof]
Recall (from Lemma~\ref{lem:jdtDes} and Lemma~\ref{lem:shapes} with $m = k = n_1$) that jdt is a $\Des$-preserving map from  $\SYT((n_1) \oplus (n_2))$  to $\TTT$. As discussed above, $(1)\oplus (n-1)$ may replace the pair of shapes $(n-1,1)$ and $(n)$, so we can consider $\hTTT$ instead of $\TTT$.

By \cite[Lemma 2.2(iii)]{ARR},
the distribution of $\cDes$ is determined by the distribution of $\Des$.
Hence, the number of SYT in $\hTTT$ with $d$ cyclic descents is equal to the number of  SYT of shape $(n_1) \oplus (n_2)$ with $d$ cyclic descents; namely,
\begin{equation}\label{eq:n1n2}
\hTTT^{\cdes}(q) = \sum_{T\in (n_1)\oplus (n_2)} q^{\cdes(T)}. 
\end{equation}

Next, recall  the $\cDes$-preserving injection $f$ from Section~\ref{sec:words}, which encodes each $T \in \SYT((n_1) \oplus (n_2))$ by a word 
$w = w_1 \cdots w_n \in [2]^n$.
This encoding is a bijection from $\SYT((n_1) \oplus (n_2))$ to the set of words $w \in [2]^n$ in which $1$ appears $n_1$ times. 
If $n \in \cDes(w)$ then $w$ has the form $1^{k_1} 2^{m_1} 1^{k_2} 2^{m_2} \cdots 1^{k_d} 2^{m_d}$, where $d = \cdes(w)$ and $k_i \ge 1$, $m_i \ge 1$ for all $1 \le i \le d$. Also, $\sum_{i=1}^d k_i = n_1$ and $\sum_{i=1}^d m_i = n_2$.
There are
\[
\binom{n_1-1}{d-1}\binom{n_2-1}{d-1}
\]
solutions to this system of equations and inequalities. 
Since all $n$ rotations of a given word have the same cyclic descent number $d$, and exactly $d$ of these words have a cyclic descent at $n$, we have to multiply the above product by $n/d$.
This conclusion is valid even if $w$ has a nontrivial cyclic symmetry.
\end{proof}

\begin{proof}[Third Proof]
We use Equation~\eqref{eq:n1n2} and the bijection $\Gamma$ (defined in Section~\ref{sec:paths}) from $\SYT((n_1) \oplus (n_2))$ to $\AAA_n^{n_2-n_1}$, 
which in this case maps $\cDes$ to $\cPeakA$. 
It will be convenient to rotate the paths in $\AAA_n^{n_2-n_1}$ and view them as paths from $(0,0)$ to $(n_1,n_2)$ with steps $N=(1,0)$ and $E=(0,1)$. Peaks are then corners $NE$, and $n$ is a cyclic peak if and only if the path starts with $E$ and ends with $N$. 
As shown in~\cite[Theorem 3.2]{BBES}, a path from $(0,0)$ to $(n_1,n_2)$ having $k$ peaks is uniquely determined by any choice of values $0\le x_1<\dots<x_k\le n_1-1$ and $1\le y_1<\dots<y_k\le n_2$; indeed, such a choice corresponds to the path with peaks at coordinates $(x_1,y_1),\dots,(x_k,y_k)$. Besides, $n$ is a cyclic peak of the path if and only if $x_1\neq 0$ and $y_k\neq n_2$.
It follows that there are $\binom{n_1}{k}\binom{n_2}{k}$ paths with $k$ peaks, and $\binom{n_1-1}{k}\binom{n_2-1}{k}$ of those have an additional cyclic peak at $n$.

Thus, the number of paths with $d$ cyclic peaks is equal to the number of paths having $d-1$ peaks and a cyclic peak at $n$, namely $\binom{n_1-1}{d-1}\binom{n_2-1}{d-1}$, plus the number of those where all $d$ cyclic peaks are regular peaks, namely $\binom{n_1}{d}\binom{n_2}{d}-\binom{n_1-1}{d}\binom{n_2-1}{d}$. It follows that
$$\hTTT^{\cdes}(q)=\binom{n_1-1}{d-1}\binom{n_2-1}{d-1}+\binom{n_1}{d}\binom{n_2}{d}-\binom{n_1-1}{d}\binom{n_2-1}{d}=\frac{n}{d}\binom{n_1-1}{d-1}\binom{n_2-1}{d-1}.\qedhere$$
\end{proof}

\medskip

\begin{remark}\label{rmk:Narayana}
For odd $n$, writing $n=2m+1$, Theorem~\ref{t.gf_cdes} states that the number of SYT in $\hTTT$ with $d$ cyclic descents 
equals $n\,N(m,d)$, where $N(m,d)$ is the Narayana number, counting, among many other things, Dyck paths of semilength $m$ with $d$ peaks. 
\end{remark}

\medskip

The proofs of Proposition~\ref{t.gf_des} and Theorem~\ref{t.gf_cdes} can be generalized in a straightforward manner, replacing $n_1$ with any $1 \le k \le n_1$ (and $n_2$ with $n-k$).  
One deduces the following two results, refining the previous results to the level of single shapes.

\begin{proposition}\label{t.gf_des_1}
For any $1 \le k \le n/2$,
\[
\sum_{T \in \SYT((n-k,k))} q^{\des(T)}
= \sum_{d=0}^{k} \left[ \binom{k}{d} \binom{n-k}{d} -\binom{k-1}{d} \binom{n-k+1}{d} \right] q^d.
\]
\end{proposition}

\begin{proof}
Lemma~\ref{lem:jdtDes} and Lemma~\ref{lem:shapes} (with $m = k \le n/2$) imply that
\[
\sum_{T \in \SYT((k) \oplus (n-k))} {\bf x}^{\Des(T)}
= \sum_{d=0}^{k} \,\sum_{T \in \SYT((n-d,d))} {\bf x}^{\Des(T)}.
\]
Subtracting the corresponding equation for $k-1$ gives 
\begin{equation}\label{eq:substract}
\sum_{T \in \SYT((k) \oplus(n-k))} {\bf x}^{\Des(T)} - \sum_{T \in \SYT((k-1) \oplus(n-k+1))} {\bf x}^{\Des(T)}
= \sum_{T \in \SYT((n-k,k))} {\bf x}^{\Des(T)}.
\end{equation}
The specialization $x_1 = \ldots = x_{n-1} = q$ together with the aforementioned generalization of Proposition~\ref{t.gf_des} yield the desired formula.
\end{proof}

\begin{theorem}\label{t.gf_cdes_1}
For any $2 \le k \le n/2$,
\[
\sum_{T \in \SYT((n-k,k))} q^{\cdes(T)}
= \sum_{d=1}^{k} \frac{n}{d} \left[ \binom{k-1}{d-1} \binom{n-k-1}{d-1} -\binom{k-2}{d-1} \binom{n-k}{d-1} \right] q^d.
\]
\end{theorem}

\begin{proof}
By \cite[Lemma 2.2(iii)]{ARR},
the cyclic descent set distribution is determined by the descent set distribution.
Hence, Equation~\eqref{eq:substract} also holds when $\Des$ is replaced by $\cDes$. The result now follows from the aforementioned generalization of Theorem~\ref{t.gf_cdes}.
\end{proof}

By a similar argument, Proposition~\ref{prop:cDes-2-rows} implies
the following more refined result.

\begin{theorem}\label{t.gf_cdes_2}
For any $2 \le k \le n/2$ and a subset $\varnothing \subsetneq J \subsetneq [n]$ of size $t$, define the cyclic differences $d_1, \ldots, d_t$ of $J$ as in Proposition~\ref{prop:cDes-2-rows}. Then
\[
\card{\{T \in \SYT((n-k,k)) \,:\, \cDes(T) = J\}}
= \coeff_{q^{k}} \left( \prod_{i=1}^{t} \sum_{j=1}^{d_i-1} q^j \right) -
\coeff_{q^{k-1}} \left( \prod_{i=1}^{t} \sum_{j=1}^{d_i-1} q^j \right).
\]
\end{theorem}

\section{Two-row skew shapes}\label{sec:skew}

In this section we give two different definitions of cyclic descent sets on skew
SYT with two rows. The first one applies the {\it jeu de taquin} construction from Section~\ref{sec:jdt} to reduce 
to straight tableaux, while the second applies the lattice path interpretation from Section~\ref{sec:2rows}.

Parameterizing two-row skew shapes as $(n-k+m,k)/(m)$, where $0\le m\le k<n$ and $2k\le n+m$,
we will consider those with $k\ge m+2$, which we call {\em non-ribbon two-row skew shapes}. For $k=m+1$, the shape $(n-k+m,k)/(m)$ is a ribbon and thus, by Proposition~\ref{prop:ribbons}, has no cyclic descent map.
For $k=m$, the shape $(n-k+m,k)/(m)$ is a strip, and this case was considered in Section~\ref{sec:strips}.

\subsection{Cyclic descents on two-row skew shapes via jeu de taquin}

\begin{definition}\label{def:cDesJ}
For a SYT $T$ of non-ribbon two-row skew shape, let
$$\cDes_J(T):=\cDes(\Jdt(T)),$$
where the right-hand side is given by
Definition~\ref{def:cDes2rowleft}.
\end{definition}

\begin{remark}
The above definition does not depend on $m$, that is, any $T\in\SYT((n-k+m,k)/(m))$ and $T'\in\SYT((n-k+m',k)/(m'))$ of non-ribbon two-row skew shape having the same entries in the first and second row satisfy $\cDes_J(T')=\cDes_J(T)$.
\end{remark}

\begin{theorem}\label{skew-2-row-thm_J} 
For every $2\le m+2\le k<n$ with $2k\le n+m$, Definition~\ref{def:cDesJ} gives a cyclic descent map for $(n-k+m,k)/(m)$.
\end{theorem}

\begin{proof}
Recall that, by Theorem~\ref{thm:2rows_intro}, Definition~\ref{def:cDes2rowleft} gives a cyclic descent map for two-row straight shapes. Let $T\in\SYT((n-k+m,k)/(m))$. By 
Definition~\ref{def:cDesJ} and Lemma~\ref{lem:jdtDes},
$$
\cDes_J(T)\cap[n-1]=\cDes(\Jdt(T))\cap[n-1]=\Des(\Jdt(T))=\Des(T),
$$
and thus $\cDes_J$ satisfies part (i) of Definition~\ref{def:extension}.

Next we show that part (ii) is also satisfied. By Lemma~\ref{lem:shapes},
\[
\{\Jdt(T):\ T\in\SYT((n-k+m,k)/(m))\}= \bigcup_{d=k-m}^{\min\{k,n-k\}} \SYT(n-d,d).
\]
Hence, since $k-m\ge2$,
\[
\multiset{\cDes_J(T):\ T\in\SYT((n-k+m,k)/(m))}=
\bigsqcup_{d=k-m}^{\min\{k,n-k\}}\multiset{\cDes(T):\ T\in 
\SYT(n-d,d)},
\]
with $\cDes$ on the right hand side given by Definition~\ref{def:cDes2rowleft}.
By Theorem~\ref{thm:2rows_intro}, each multiset in this union is closed under cyclic rotation modulo $n$. 
\end{proof}

\subsection{Cyclic descents on two-row skew shapes via lattice paths}

A second definition of cyclic descents on two-row skew shapes is obtained by
generalizing the construction in Section~\ref{sec:2rows}, considering paths in $\AAA_n$ that do not go below a given horizontal line, and extending the injections $\varphi$ and $\psi$ to more general sets of paths.

Let $\PPP_{n,m}$ be the set of paths in $\AAA_n$ that do not go below the line $y=-m$, and let $\PPP_{n,m}^h$ be the subset of those that end at height $h$. Let $\QQQ_{n,m}=\AAA_n\setminus\PPP_{n,m}$, and define $\QQQ_{n,m}^h$ similarly. Note that, by definition, $\PPP^h_{n,0}=\PPP^h_n$ and $\QQQ^h_{n,0}=\QQQ^h_n$.

Fix $0\le m\le k<n$ with $2k\le n+m$. The map $\Gamma$ defined in Section~\ref{sec:paths} gives a bijection between $\SYT((n-k+m,k)/(m))$ and $\PPP_{n,m}^{n-2k}$, mapping the statistic $\Des$ on tableaux to the statistic $\Peak$ on paths.
Again, we restrict to non-ribbon two-row skew shapes by further assuming in the rest of this section that $k\ge m+2$.

As before, finding a cyclic descent map for $(n-k+m,k)/(m)$ as in Definition~\ref{def:extension} is equivalent to finding a cyclic peak map $\cPeakm$ on $\PPP_{n,m}^{n-2k}$ as in Definition~\ref{def:cPeak}. 
Indeed, once $\cPeakm$ has been defined, letting $\cDes_P(T):=\cPeakm(\Gamma(T))$ for $T\in\SYT((n-k+m,k)/(m))$ will yield a cyclic descent map for $(n-k+m,k)/(m)$.

The following lemma generalizes Lemma~\ref{lem:varphi}. Note that the condition $k\neq m$ guarantees that the set $\AAA_n^{n-2k+2m+2}$ is not empty.

\begin{lemma}\label{lem:varphi_m} There is a $\Peak$-preserving bijection
$$\varphi_m:\AAA_n^{n-2k+2m+2}\longrightarrow \QQQ_{n,m}^{n-2k}\subset \AAA_n^{n-2k}.$$
\end{lemma}

\begin{proof} Given $P\in\AAA_n^{n-2k+2}$, match $U$ and
$D$ steps that face each other, as in the proof of Lemma~\ref{lem:varphi}. Then take the leftmost $m+1$ unmatched $U$ steps (which must
exist because $n-2k+2m+2\ge m+1$) and turn them into $D$ steps. Let $\varphi_m(P)$
be the resulting path. 

The image of $\varphi_m$ consists of those paths in $\AAA_n^{n-2k}$ that have at least $m+1$
unmatched $D$ steps, that is, those in $\QQQ_{n,m}^{n-2k}$.
The preimage of a path in $\QQQ_{n,m}^{n-2k}$ is obtained by matching $U$
and $D$ steps that face each other, and then changing the rightmost $m+1$ unmatched $D$ steps into $U$ steps.
\end{proof}

The set of paths in $\AAA_n^{n-2k}$ that are not in the image of $\varphi_m$ is $\PPP_{n,m}^{n-2k}$. Since $\multiset{\cPeakA(P):P\in\AAA_n^h}$ is closed under cyclic rotation, 
requiring that 
\begin{equation}\label{eq:Peak'm}
|\{P\in\PPP_{n,m}^{n-2k}:\cPeakm(P)=J\}|=|\{P\in\AAA_n^{n-2k}:\cPeakA(P)=J\}|-|\{P\in\AAA_n^{n-2k+2m+2}:\cPeakA(P)=J\}|
\end{equation}
for every $J\subseteq[n]$ is sufficient for $\cPeakm$ on $\PPP_{n,m}^{n-2k}$ to satisfy condition (ii) in Definition~\ref{def:cPeak}.
Additionally, provided that $\cPeakm$ extends $\Peak$, Lemma~\ref{lem:varphi_m} implies that it is enough to require condition~\eqref{eq:Peak'm} for sets 
$J=I\sqcup\{n\}$ with $I\subseteq[n-1]$.

Fix $I\subseteq[n-1]$. Paths in $\AAA_n^{n-2k}$ with peak set $I$ can be partitioned as follows:
\begin{multline}\label{eq:unionm}
\{P\in\AAA_n^{n-2k}:\Peak(P)=I\}\ =\{P\in\PPP_{n,m}^{n-2k}:\Peak(P)=I\}\sqcup\{P\in\QQQ_{n,m}^{n-2k}:\Peak(P)=I\}\\
=\{P\in\PPP_{n,m}^{n-2k}:\cPeakA(P)=I\}\sqcup\{P\in\PPP_{n,m}^{n-2k}:\cPeakA(P)=I\sqcup\{n\}\}\\
\sqcup\{P\in\QQQ_{n,m}^{n-2k}:\cPeakA(P)=I\}\sqcup\{P\in\QQQ_{n,m}^{n-2k}:\cPeakA(P)=I\sqcup\{n\}\}.
\end{multline}

By Lemma~\ref{lem:varphi_m}, $\varphi_m$ gives a bijection between
$\{P\in\AAA_n^{n-2k+2m+2}:\Peak(P)=I\}$ and
$\{P\in\QQQ_{n,m}^{n-2k}:\Peak(P)=I\}$.
The next lemma, whose proof is almost identical to that of Lemma~\ref{lem:cPeak},
describes how the statistic $\cPeakA$ is affected by this bijection. The paths with $\cPeakA(\varphi_m(P))=\cPeakA(P)\sqcup\{n\}$ correspond to the striped region in Figure~\ref{fig:varphi_m_blocks}.

\begin{figure}[htb]
\centering
\begin{tikzpicture}[scale=1.1]
\draw (-.5,0) rectangle (6,2.5);
\fill[blue!30!white,postaction={pattern=north east lines}] (-.5,1) rectangle (6,1.5);
\fill[blue!30!white] (-.5,1.5) rectangle (6,2.5);
\draw (-.5,1) rectangle (6,2.5);
\filldraw[fill=yellow!40!white] (-.5,0) rectangle (6,1);
\draw (2.75,1.75) node {$\{P\in\AAA_n^{n-2k+2m+2}:\cPeakA(P)=I\}$};
\draw (2.75,.5) node {$\{P\in\AAA_n^{n-2k+2m+2}:\cPeakA(P)=I\sqcup\{n\}\}$};
\draw [decorate,decoration={brace,mirror,amplitude=5pt},xshift=4pt] (6,0) -- (6,2.5);
\draw[->] (6.5,1.25)-- node[above]{$\varphi_m$} node[below]{$\sim$} (7.5,1.25);
\draw (3,3.25) node {$\{P\in\AAA_n^{n-2k+2m+2}:\Peak(P)=I\}$};

\begin{scope}[shift={(8,0)}]
\draw [decorate,decoration={brace,amplitude=5pt},xshift=-4pt] (0,0) -- (0,2.5);
\draw (0,0) rectangle (6,4.5);
\filldraw[fill=red!30!white] (0,3.5) rectangle (6,4.5);
\filldraw[fill=orange!40!white] (0,2.5) rectangle (6,3.5);
\filldraw[fill=blue!30!white] (0,1.5) rectangle (6,2.5);
\fill[yellow!40!white,postaction={pattern=north east lines}] (0,1) rectangle (6,1.5);
\fill[yellow!40!white] (0,0) rectangle (6,1);
\draw (0,0) rectangle (6,1.5);
\draw (3,4) node {$\{P\in\PPP_{n,m}^{n-2k}:\cPeakA(P)=I\}$};
\draw (3,3) node {$\{P\in\PPP_{n,m}^{n-2k}:\cPeakA(P)=I\sqcup\{n\}\}$};
\draw (3,2) node {$\{P\in\QQQ_{n,m}^{n-2k}:\cPeakA(P)=I\}$};
\draw (3,.75) node {$\{P\in\QQQ_{n,m}^{n-2k}:\cPeakA(P)=I\sqcup\{n\}\}$};
\draw (3,5.25) node {$\{P\in\AAA_n^{n-2k}:\Peak(P)=I\}$};
\end{scope}
\end{tikzpicture}
\caption{The behavior of $\varphi_m$ with respect to $\cPeakA$.}
\label{fig:varphi_m_blocks}
\end{figure}

\begin{lemma}\label{lem:cPeakm}
For $P\in \AAA_n^{n-2k+2m+2}$, 
$$\cPeakA(\varphi_m(P))=\begin{cases} \cPeakA(P)\sqcup\{n\} & \text{if $P$ is of the form $P=URU$ for some $R\in\PPP_{n-2}^{n-2k+2m}$},\\
\cPeakA(P) & \text{otherwise.} \end{cases}$$
\end{lemma}

Using the decomposition~\eqref{eq:unionm} and Lemma~\ref{lem:cPeakm}, Equation~\eqref{eq:Peak'm} for $J=I\sqcup\{n\}$ can be written as
\begin{multline}\label{eq:sizem}
|\{P\in\PPP_{n,m}^{n-2k}:\cPeakm(P)=I\sqcup\{n\}\}|\\
=|\{P\in\AAA_n^{n-2k}:\cPeakA(P)=I\sqcup\{n\}\}|-|\{P\in\AAA_n^{n-2k+2m+2}:\cPeakA(P)=I\sqcup\{n\}\}|\\
=|\{P\in\PPP_{n,m}^{n-2k}:\cPeakA(P)=I\sqcup\{n\}\}|\,+\,|\{P: P=URU \text{ for some } R\in\PPP_{n-2}^{n-2k+2m} \text{ and } \Peak(P)=I\}|.
\end{multline}

For paths $P\in\PPP_{n,m}^{n-2k}$ with $\cPeakA(P)=I\sqcup\{n\}$, we simply define $\cPeakm(P)=I\sqcup\{n\}$. For Equation~\eqref{eq:sizem} to be satisfied, we will find a subset of $\{P\in\PPP_{n,m}^{n-2k}:\cPeakA(P)=I\}$ of size equal to the last summand of the right-hand side of \eqref{eq:sizem} (corresponding to the striped region in Figure~\ref{fig:varphi_m_blocks}); and then let $\cPeakm(P):=I\cup\{n\}$ for $P$ in the subset, and $\cPeakm(P):=I$ for $P$ outside of the subset (but with $\cPeakA(P)=I$).
The subset will be the image of an injection 
$$\psi_m:\{P: P=URU \text{ for some } R\in\PPP_{n-2}^{n-2k+2m} \text{ and } \Peak(P)=I\}\longrightarrow \{P\in\PPP_{n,m}^{n-2k}:\cPeakA(P)=I\},$$
that we define next.

Given $P$ in the domain, we can uniquely write $P=UQDU^r$ for some path $Q$ and $r\ge1$. Indeed, $R$ must contain a $D$, since $k\ge m+2$ implies that $n-2k+2m<n-2$. Consider two cases:
\begin{enumerate}[(a)]
\item If $r\ge m+1$, let $\psi_m(P)=UQD^{m+2}U^{r-m-1}$. 
\item If $r<m+1$, let $\psi_m(P)$ be the path obtained from $UQD$ by matching $U$ and $D$ steps as usual and changing the leftmost $m+1-r$ unmatched $U$ steps into $D$ steps, then appending $D^r$. 
\end{enumerate}
Figure~\ref{fig:psi} shows an example of this construction in each case. 

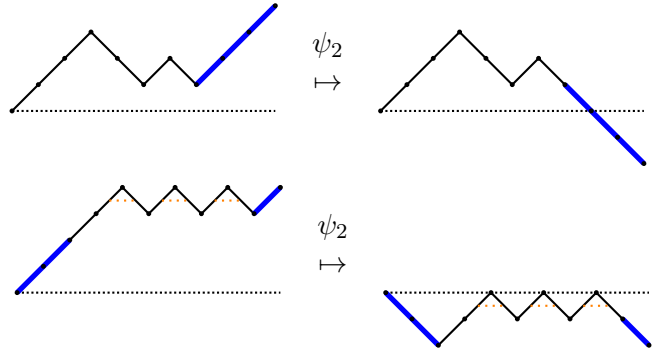
\begin{figure}[htb]
\centering
\begin{tikzpicture}[thick, scale=0.35]
    \def\U{(1,1)}     \def\D{(1,-1)}
    \draw (0,0) coordinate(d0)
      -- ++\U coordinate(d1)
      -- ++\U coordinate(d2)
      -- ++\U coordinate(d3)
      -- ++\D coordinate(d4)
      -- ++\D coordinate(d5) 
      -- ++\U coordinate(d6)
      -- ++\D  coordinate(d7)
      -- ++\U coordinate(d8)
      -- ++\U coordinate(d9)
      -- ++\U coordinate(d10);
     \draw[densely dotted] (0,0)--(10,0);
      \draw[blue,line width=2.0pt] (d7)--(d10);
      \foreach \x in {0,...,10} {
        \draw (d\x) circle (1.5pt);
      }
      \draw (12,1) node [label=$\psi_2$] {$\mapsto$};
      \begin{scope}[shift={(14,0)}]
    \draw (0,0) coordinate(d0)
      -- ++\U coordinate(d1)
      -- ++\U coordinate(d2)
      -- ++\U coordinate(d3)
      -- ++\D coordinate(d4)
      -- ++\D coordinate(d5) 
      -- ++\U coordinate(d6)
      -- ++\D  coordinate(d7)
      -- ++\D coordinate(d8)
      -- ++\D coordinate(d9)
      -- ++\D coordinate(d10);
     \draw[densely dotted] (0,0)--(10,0);
      \draw[blue,line width=2.0pt] (d7)--(d10);
      \foreach \x in {0,...,10} {
        \draw (d\x) circle (1.5pt);
      }
      \end{scope}
    \end{tikzpicture}  \medskip

  \begin{tikzpicture}[thick, scale=0.35]
    \def\U{(1,1)}     \def\D{(1,-1)}
    \draw (0,0) coordinate(d0)
      -- ++\U coordinate(d1)
      -- ++\U coordinate(d2)
      -- ++\U coordinate(d3)
      -- ++\U coordinate(d4)
      -- ++\D coordinate(d5) 
      -- ++\U coordinate(d6)
      -- ++\D  coordinate(d7)
      -- ++\U coordinate(d8)
      -- ++\D coordinate(d9)
      -- ++\U coordinate(d10);
     \draw[densely dotted] (0,0)--(10,0);
      \draw[dotted,orange] (3.5,3.5)--(4.5,3.5);
      \draw[dotted,orange] (5.5,3.5)--(6.5,3.5);
      \draw[dotted,orange] (7.5,3.5)--(8.5,3.5);
      \draw[blue,line width=2.0pt] (d0)--(d2);
      \draw[blue,line width=2.0pt] (d9)--(d10);
      \foreach \x in {0,...,10} {
        \draw (d\x) circle (1.5pt);
      }
      \draw (12,1) node [label=$\psi_2$] {$\mapsto$};
      \begin{scope}[shift={(14,0)}]
    \draw (0,0) coordinate(d0)
      -- ++\D coordinate(d1)
      -- ++\D coordinate(d2)
      -- ++\U coordinate(d3)
      -- ++\U coordinate(d4)
      -- ++\D coordinate(d5) 
      -- ++\U coordinate(d6)
      -- ++\D  coordinate(d7)
      -- ++\U coordinate(d8)
      -- ++\D coordinate(d9)
      -- ++\D coordinate(d10);
     \draw[densely dotted] (0,0)--(10,0);
      \draw[dotted,orange] (3.5,-.5)--(4.5,-.5);
      \draw[dotted,orange] (5.5,-.5)--(6.5,-.5);
      \draw[dotted,orange] (7.5,-.5)--(8.5,-.5);
      \draw[blue,line width=2.0pt] (d0)--(d2);
      \draw[blue,line width=2.0pt] (d9)--(d10);
      \foreach \x in {0,...,10} {
        \draw (d\x) circle (1.5pt);
      }
      \end{scope}
    \end{tikzpicture}
\caption{Two examples of the injection $\psi_m$ with $m=2$, $n=10$ and $k=6$.}
\label{fig:psi2}
\end{figure}

It is clear from the construction that $\psi_m$ preserves the statistic $\Peak$, and that it never creates a cyclic peak at $n$, since $\psi_m(P)$ always either starts with a $U$ (in case (a)) or ends with a $D$ (in case (b)). Also, since $\psi_m$ changes $m+1$ $U$ steps into $D$ steps, the ending height of $\psi_m(P)$ is $n-2k$, and this path never goes below $y=-m$, so $\psi_m(P)\in\PPP_{n,m}^{n-2k}$.

Finally, to check that $\psi_m$ is injective, we describe how to recover $P$ from $P'=\psi_m(P)$. If $P'$ starts with a $U$, then it was obtained from case (a), and so it is enough to change the last $m+1$ $D$ steps of the path into $U$ steps, recovering $P$. If $P'$ starts with a $D$, then it was obtained from case (b). In this case we can find $r$ because the lowest point of $P'$ not including the ending run of $D$s is at height $-(m+1-r)$. Then, $P$ is recovered by first removing the $D^r$ at the end of $P'$, then matching $U$ with $D$ steps as usual, and finally changing the $m+1-r$ unmatched $D$ steps into $U$ steps.

To summarize, we have constructed a cyclic peak map $\cPeakm$ on $\PPP_{n,m}^{n-2k}$ satisfying Definition~\ref{def:cPeak} by letting $n\in\cPeakm(P)$ if and only if either $n\in\cPeakA(P)$, or $P$ is in the image of $\psi_m$ for some $I$.
This image consists of paths in $\PPP_{n,m}^{n-2k}$ that either \begin{enumerate}[(A)]
\item start with a $U$, end with $D^{m+2}U^j$ for some $j\ge0$, and after changing this ending to $DU^{j+m+1}$ they are of the form $URU$ for some $R\in\PPP_{n-2}$; or
\item start with a $D$ and end with $D^{\ell+m+2}$, where $\ell$ is the height of their lowest point not including the ending run of $D$s (note that ending with $D^{\ell+m+2}$ is equivalent to going below $y=-(m-b)$ where $b+1$ is the number of consecutive $D$s at the end of the path).
\end{enumerate}

\medskip

Via the bijection $\Gamma$, the statistic $\cPeakm$ translates into a statistic $\cDes_P$ on $\SYT((n-k+m,k)/(m))$ that we define next.
For $T\in\SYT((n-k+m,k)/(m))$, let $b=b(T)$ be the
number of consecutive entries at the end of the second row of $T$
minus one. In other words, $b\ge 0$ is the largest number so that
the second row of $T$ ends in $i-b,\dots,i-1,i$, for some $i$.

\begin{definition}\label{def:cDesP}
For $T\in\SYT((n-k+m,k)/(m))$ with $2\le m+2\le k<n$ and $2k\le n+m$, let
$n\in\cDes_P(T)$ if and only if at least one of the following
conditions holds:
\begin{enumerate}[(1)]
\item $1$ is lower than $n$ in $T$; \item $1$ is in the first row,
$b>m$, and for every $1\le i\le k-b-1$, we have $T_{2,i} >
T_{1,i+m+1}$; \item $1$ is in the second row and $b>m$; \item $1$
is in the second row, $b\le m$, and there exists some $m-b+1 \le
i\le k-b-1$ such that $T_{2,i} < T_{1,i+b}$.
\end{enumerate}
\end{definition}

\begin{example} According to Definition~\ref{def:cDesP}, we have $10\in\cDes\left(\young(::234567,189\ten)\right)$ because condition (4) is satisfied, but $10\notin\cDes\left(\young(:134567,289\ten)\right)$.
\end{example}

\begin{theorem}\label{skew-2-row-thm_S} 
For every $2\le m+2\le k<n$ and $2k\le n+m$, Definition~\ref{def:cDesP} gives a cyclic descent map for $(n-k+m,k)/(m)$.
\end{theorem}

\begin{proof}
Since we have proved that $\cPeakm$ is a cyclic peak map on $\PPP_{n,m}^{n-2k}$, 
it is enough to show that, for every $T\in \SYT((n-k+m,k)/(m))$, we have $\cDes_P(T)=\cPeakm(P)$, where $P=\Gamma(T)$.

Recall that $n\in\cPeakm(P)$ precisely in the following two scenarios:
\begin{itemize} 
\item 
$n\in\cPeakA(P)$. This case corresponds to $n$ being in the first row of $T$ and $1$ being in the second row, which is condition (1).
\item 
$P$ is in the image of $\psi_m$, which means that cases (A) or (B) above apply.

In case (A), $1$ is in the first row of $T$, the last $m+2$ entries of the second row are consecutive (equivalently, $b>m$), and if we remove $1$ and the last $m+1$ entries (or equivalently, the last $b+1$ entries) of the second row and we left-justify the rows, the resulting tableau is standard. This is equivalent to condition (2).

In case (B), $1$ is in the second row of $T$, and the condition of the path going below $y=-(m-b)$ translates as follows: if we remove the last $b+1$ entries of the second row of $T$ and shift the first row $b$ positions to the left, we no longer have a standard tableau. This is equivalent to conditions (3) and (4). \qedhere
\end{itemize}
\end{proof}

\medskip

\begin{remark}\label{rem:skew}\
\begin{enumerate}
\item 
The above two definitions of cyclic descent set on skew
tableaux do not coincide in general. For example, for $$T=\young(:246,135),\qquad
\cDes_P(T)=\{2,4,6\}\ \ \text{and}\ \ \cDes_J(T)=\{2,4\}.$$ On the other
hand, for
\[
T'=\young(:124,356),\qquad \cDes_P(T')=\{2,4\}\ \ \text{and}\ \
\cDes_J(T')=\{2,4,6\}.
\]

\item 
Examples of special interest are the disconnected two-row shapes, that is, the horizontal strips $(n-k+m,k)/(m)$ with $m\ge k$.
SYT of this shape may be identified with subsets of
$[n]$ of size $k$. Even though Definition~\ref{def:cDesP} requires $m+2\le k$, the definition of
$\cDes_P$ that would result by setting $m=k$ coincides with the definition of $\cDes$ on two-row horizontal strips 
(Definition~\ref{def-hor}). Indeed, for $m=k$, conditions (2), (3) and (4) in Definition~\ref{def:cDesP} never hold.
On the other hand, $\cDes_J$ from Definition~\ref{def:cDesJ} for $m=k$ does not coincide with $\cDes$ from Definition~\ref{def-hor}. 
For example, 
for $T=\young(::12,34)$, \ $\cDes_P(T)=\cDes(T)=\{2\}$ while $\cDes_J(T)=\{2,4\}$.

\item 
Both $\cDes_J$ and $\cDes_P$ coincide with $\cDes$ from Definition~\ref{def:cDes2rowleft} on
straight shapes, and in particular with Rhoades' definition~\cite{Rhoades} on two-row rectangular shapes.
\end{enumerate}

\end{remark}

It is shown in~\cite[Lemma 2.2]{ARR} that, under mild conditions, the distribution of the cyclic descent set 
over $\SYT(\lambda/\mu)$ is uniquely determined by $\lambda/\mu$. 
Since both maps $\cDes_P$ and $\cDes_J$ satisfy the conditions of~\cite[Lemma 2.2]{ARR}
we deduce the following.

\begin{proposition}\label{cor_equid}
The statistics $\cDes_P$ and $\cDes_J$ are equidistributed over
$\SYT((n-k+m,k)/(m))$.
\end{proposition}

\section{Final remarks and open problems}\label{sec:final}

After the results in this paper, explicit cyclic descent maps are now known for rectangles \cite{Rhoades}, straight shapes with a disconnected box added in the upper-right corner \cite{ER16}, strips, hooks plus one cell, two-row straight shapes and two-row skew shapes. These results are summarized in Table~\ref{tab:cDes}. 

\begin{table}[htb]
\centering
\begin{tabular}{c|c|c|c}
Shape & Example & $n\in\cDes(T)$ if and only if & Result in the paper \\
\hline\hline
\begin{tabular}{c}
rectangle \\ $(m^{n/m})$ 
\end{tabular} & $\young(\blank\blank\blank,\blank\blank\blank)$ & $n-1\in \Des (\jdt(-1+T))$ & Theorem~\ref{Rhoades_thm2} \\[3mm]
\hline
$\lambda^\Box$ & $\young(:::\blank,\blank\blank\blank,\blank\blank)$ & 
\begin{tabular}{c}
 $n$ is strictly north of $1$, or \\
$n-d\in \Des (\jdt(-d+T))$, \\
where $d$ is the letter\\ in the disconnected cell
\end{tabular} &
Theorem~\ref{thm:cDes_lambda_box} \\[4mm]
\hline
\begin{tabular}{c} hook plus an \\ internal cell \\ $(n-k,2,1^{k-2})$ \\ $2\le k\le n-2$  \end{tabular} & $\young(\hfill\hfill\hfill\hfill,\hfill\hfill,\hfill)$ & $T_{2,2}-1$ is in first column & Theorem~\ref{thm:hooks_plus1_intro}
\\[4mm]
\hline
strip & $\young(:::\blank\blank\blank,::\blank,::\blank,\blank\blank)$ & 
\begin{tabular}{c}  $n-1\in \Des (\jdt(-1+T))$; \\
equivalently, $n$ is strictly \\ north or weakly east of $1$ \end{tabular} & Proposition~\ref{def-strip}\\[7mm]
\hline
\begin{tabular}{c} 
two-row shape \\ $(n-k,k)$ \\ $2 \le k\le n/2$  \end{tabular}& $\young(\blank\blank\blank\blank,\blank\blank)$ & \begin{tabular}{c} $T_{2,k}=T_{2,k-1} + 1$, and\\
$T_{2,i-1} > T_{1,i} \quad (\forall 1 < i < k)$ \end{tabular} & Theorem~\ref{thm:2rows_intro} 
\\
\hline
\begin{tabular}{c} 
two-row skew shape \\ $(n-k+m,k)/(m)$ \\ $k\neq m+1$  \end{tabular}& $\young(:\blank\blank\blank,\blank\blank\blank)$ & (see Definitions~\ref{def:cDesJ} and~\ref{def:cDesP}) & \begin{tabular}{c} Theorems \ref{skew-2-row-thm_J}\\ and~\ref{skew-2-row-thm_S} \end{tabular}
\end{tabular}
\caption{Shapes for which a cyclic descent 	map is known explicitly.}
\label{tab:cDes}
\end{table}

In addition, a cyclic descent map on increasing semistandard Young tableaux of rectangular shape, which was recently introduced by Dilks, Pechenik and Striker~\cite{DPS},
can be transferred to $\SYT(k,k,1^{n-2k})$ using a descent-preserving bijection of Pechenik~\cite{Pechenik}.

\medskip

It is proved in~\cite{ARR}, in a nonconstructive fashion, that cyclic descent maps exist for all non-ribbon skew shapes. It remains an open problem to explicitly define $\cDes$ for shapes not listed above.

\begin{problem}\label{prob:cDes}
Find an explicit combinatorial definition of a cyclic descent map $\cDes$ for any non-ribbon shape $\lambda/\mu$.
\end{problem}

Whereas for certain shapes, such as hooks plus one cell, there exists a unique cyclic descent map (Theorem~\ref{thm:hooks_plus1_intro}), for other shapes, such as two-row skew shapes, different definitions of $\cDes$ are possible (Theorems~\ref{skew-2-row-thm_J}  and~\ref{skew-2-row-thm_S}).
In these cases, one may ask if one definition is more natural than the others. 

The combinatorial descriptions of $\cDes$ in the cases known so far vary significantly from shape to shape. For example, our definition for strips given in Proposition~\ref{prop:ribbons} determines whether $n$ is a cyclic descent by simply comparing the positions of $n$ and $1$ in the SYT, whereas the definitions for hooks plus one cell (Theorem~\ref{thm:hooks_plus1_intro}) and two-row shapes (Theorem~\ref{thm:2rows_intro}) do not explicitly consider the position of $n$.
Ideally, one would hope for a unified definition of $\cDes$ for all non-ribbon shapes, or at least for the straight ones.

\medskip

For some shapes, namely rectangles \cite{Rhoades}, straight shapes with a disconnected box added in the upper-right corner \cite{ER16}, strips, hooks plus one cell, and two-row straight shapes, we have been able to explicitly describe a bijection $\bij:\SYT(\lambda/\mu)\to\SYT(\lambda/\mu)$ with the property that $\cDes(\bij T)=1+\cDes(T)$ for every $T\in\SYT(\lambda/\mu)$.
In general, finding such a bijection is not immediate, even after a combinatorial definition of $\cDes$ has been found; for example, we have no explicit $\bij$ for two-row skew shapes. 

The bijection $\bij$ generates a $\bbz$-action on $\SYT(\lambda/\mu)$, which, in some of the above cases, is in fact a $\bbz_n$-action, where $n$ is the number of cells of $\lambda/\mu$. This is the case for rectangles and strips (where $\bij$ is simply the inverse promotion operator), 
for straight shapes with a disconnected box added in the upper-right corner, and for hooks plus one cell.

\begin{problem}
For non-ribbon shapes $\lambda/\mu$ for which Problem~\ref{prob:cDes} is solved,
describe an explicit bijection $\bij$ on $\SYT(\lambda/\mu)$ which cyclically shifts $\cDes$ and, ideally, generates a $\bbz_n$-action. 	
\end{problem}

Again, it would be desirable to have a unified definition of the bijection $\bij$ for all non-ribbon shapes. We point out that the descriptions of $\bij$ in the above known cases rely on variations of {\it jeu de taquin}.

\medskip

An intriguing problem is to find statistics which are equidistributed with cyclic descents on SYT of a given shape. Instances of such statistics for strips and two-row shapes are 
described in Sections~\ref{sec:words} and~\ref{sec:combinatorial2row}; see also Remark~\ref{rmk:Narayana}.

\begin{problem}
Find statistics on combinatorial objects 
which are equidistributed
with cyclic descents on SYT of given shapes. 
\end{problem}

For example, the number of SYT with $n$ cells and at most three rows is known to be equal to the $n$-th Motzkin number~\cite{Regev}. Finding statistics on Motzkin paths which are equidistributed with cyclic descent sets of SYT may help solve Problem~\ref{prob:cDes} for three-row shapes.

\medskip

We conclude with two more problems of a more basic nature, motivated by the question of $\cDes$ uniqueness. They concern the set $D(\s) = \{\card{\Des(T)} \,:\, T \in \SYT(\s)\}$ of possible descent numbers for a skew shape $\s$.

\begin{problem}
By Lemma~\ref{t.unique_cDes},
the condition $D(\s) = \{k-1,k\}$ is sufficient for the uniqueness of a cyclic descent map for $\s$. 
Is it also necessary?
\end{problem}

\begin{problem}
Is $D(\s)$ an interval, consisting of consecutive integers, for any skew shape $\s$?
\end{problem}



\begin{thebibliography}{10}

\bibitem{ARR} 
R.\ M.\ Adin, V.\ Reiner and Y.\ Roichman, 
On cyclic descents of standard Young tableaux, 
preprint 2017, {\tt arXiv:1710.06664}. 

\bibitem{ARS} 
C.\ Ahlbach, B.\ Rhoades and J.\ P.\ Swanson, 
Euler-Mahonian refined cyclic sieving, 
in preparation.

\bibitem{BBES}
M.\ Barnabei, F.\ Bonetti, S.\ Elizalde and M.\ Silimbani, 
Descent sets on $321$-avoiding involutions and hook decompositions of partitions,
{\em J.\ Combin.\ Theory Ser.\ A} {\bf 128} (2014), 132--148.

\bibitem{BR} 
A.\ Berele and A.\ Regev, 
Hook Young diagrams with applications to combinatorics and to representations of Lie superalgebras, 
{\em Adv.\ Math.}~{\bf 64} (1987), 118--175.

\bibitem{Cellini} 
P.\ Cellini,
Cyclic Eulerian elements, 
{\em European J.\ Combin.}~{\bf 19} (1998), 545--552.

\bibitem{DPS} 
K.\ Dilks, O.\ Pechenik and J.\ Striker,   
Resonance in orbits of plane partitions and increasing tableaux, 
{\em J. Combin.\ Theory Ser.\ A} {\bf 148} (2017), 244--274. 

\bibitem{Dilks} 
K.\ Dilks, T.\ K.\ Petersen and J.\ R.\ Stembridge,
Affine descents and the Steinberg torus,
{\em Adv.\ Appl.\ Math.}~{\bf 42} (2009), 423--444.

\bibitem{Doran}
W.\ F.\ Doran,
A plethysm formula for $p_\mu(\ul{x})\circ h_\mu(\ul{x})$,
{\em Electron.\ J.\ Combin.}~{\bf 4} (1997), Research Paper 14, 10 pp.

\bibitem{ER16} 
S.\ Elizalde and Y.\ Roichman, 
On rotated Schur-positive sets, 
{\em J.\ Combin.\ Theory Ser.\ A}~{\bf 152} (2017), 121--137. 

\bibitem{FK} 
B.\ Fontaine and J.\ Kamnitzer, 
{\em Cyclic sieving, rotation, and geometric representation theory},
Selecta Math.~{\bf 20} (2014), 609--625. 


\bibitem{Pechenik} 
O.\ Pechenik, 
Cyclic sieving of increasing tableaux and small Schr\"oder paths,
{\em J.\ Combin.\ Theory Ser.\ A}~{\bf 125} (2014), 357--378.

\bibitem{Petersen}
T.\ K.\ Petersen, 
Cyclic descents and $P$-partitions, 
{\em J.\ Algebraic Combin.}~{\bf 22} (2005), 343--375.

\bibitem{PPR} 
T.\ K.\ Petersen, P.\ Pylyavskyy and B.\ Rhoades, 
Promotion and cyclic sieving via webs,  
{\em J.\ Algebraic Combin.}~{\bf 30} (2009), 19--41.

\bibitem{PS}  
T.\ K.\ Petersen and L.\ Serrano, 
Cyclic sieving for longest reduced words in the hyperoctahedral group, 
{\em Electron.\ J.\ Combin.}~{\bf 17} (2010), \#R67, 12 pp.

\bibitem{PR} 
J.\ Propp and T.\ Roby,  
Homomesy in products of two chains,
{\em Electron.\ J.\ Combin.}~{\bf 22} (2015), Paper 3.4, 29 pp. 

\bibitem{Regev} 
A.\ Regev, 
Asymptotic values for degrees associated with strips of Young diagrams, 
{\em Adv.\ Math.}~{\bf 41} (1981), 115--136.

\bibitem{RS} 
A.\ Regev and T.\ Seeman, 
Shuffle invariance of the super-RSK algorithm, 
{\em Adv.\ Appl.\ Math.}~{\bf 28} (2002), 59--81.

\bibitem{Remmel} 
J.\ B.\ Remmel, 
The combinatorics of $(k,l)$-hook Schur functions,
{\em Contemp.\ Math.}~{\bf 34} (1984), 253--287.

\bibitem{Rhoades} 
B.\ Rhoades, 
Cyclic sieving, promotion, and representation theory, 
{\em J.\ Combin.\ Theory Ser.\ A}~{\bf  117} (2010), 38--76.

\bibitem{EC2} 
R.\ P.\ Stanley, 
{\em Enumerative combinatorics, Vol.\ 2},
Cambridge Studies in Adv.\ Math., no.\ 62, Cambridge Univ.\ Press, Cambridge, 1999.

\bibitem{StanPE}
R.\ P.\ Stanley, Promotion and evacuation,
{\em Electron.\ J.\ Combin.}~{\bf 16} (2009), \#R67, 24 pp.

\end{thebibliography}
\end{document}